\documentclass[12pt,reqno]{amsart}

\usepackage[margin=1in]{geometry}
\usepackage{amsmath,amssymb,amsthm,mathtools}
\usepackage{mathrsfs}
\usepackage{enumitem}
\usepackage{hyperref}
\usepackage{microtype}
\usepackage{cite}
\numberwithin{equation}{section}
\allowdisplaybreaks

\hypersetup{colorlinks=true,linkcolor=blue,citecolor=blue,urlcolor=blue}

\newcommand{\T}{\mathbb{T}}
\newcommand{\R}{\mathbb{R}}
\newcommand{\N}{\mathbb{N}}

\newcommand{\HN}{\mathcal{H}_N}
\newcommand{\dd}{\,\mathrm{d}}

\newcommand{\eps}{\varepsilon}

\newcommand{\diag}{\mathrm{diag}}
\newcommand{\tr}{\mathrm{tr}}
\newcommand{\HS}{\mathrm{HS}}

\newcommand{\I}{\mathrm{\uppercase\expandafter{\romannumeral1}}}
\newcommand{\II}{\mathrm{\uppercase\expandafter{\romannumeral2}}}
\newcommand{\III}{\mathrm{\uppercase\expandafter{\romannumeral3}}}
\newcommand{\IV}{\mathrm{\uppercase\expandafter{\romannumeral4}}}

\newtheorem{theorem}{Theorem}[section]
\newtheorem{proposition}[theorem]{Proposition}
\newtheorem{lemma}[theorem]{Lemma}
\newtheorem{corollary}[theorem]{Corollary}
\theoremstyle{definition}
\newtheorem{definition}[theorem]{Definition}

\theoremstyle{remark}
\newtheorem{remark}[theorem]{Remark}

\title[Propagation of Chaos and Gaussian Fluctuations for Coupled Maps]{Propagation of Chaos and Gaussian Fluctuations for Mean-field Coupled Maps}

\author[Cheng]{Ruicheng Cheng}
\address[Ruicheng Cheng]{\newline School of Mathematical Sciences, \newline
Peking University, Beijing, 100871, the People's Republic of China}
\email{chengruicheng02@stu.pku.edu.cn}

\date{}

\begin{document}

\maketitle

\begin{abstract}
    We study a class of discrete-time interacting particle systems arising from mean-field coupled maps. We establish a Dobrushin-type estimate, a relative entropy estimate and Central Limit Theorem (CLT) type results for this class of systems. First we control the growth of the 1-Wasserstein distance between the empirical measures and the limiting distributions. Then we use the relative entropy method to show the propagation of chaos. Finally we consider the asymptotic behavior of the fluctuations for the empirical measures and prove that the sequence of fluctuation processes converges in distribution to some Gaussian process, where we establish both qualitative and quantitative results.
\end{abstract}

\section{Introduction}\label{sec:introduction}
We consider a system of $N$ coupled maps on the one-dimensional torus $\T=\R/\mathbb{Z}$:
\begin{equation}\label{eq:particle-system}
 x_i(t+1)
 =F\left(x_i(t)+\frac{\Delta}{N}\sum_{j=1}^N h(x_i(t),x_j(t))\right),
 \qquad i=1,\dots,N,
\end{equation}
where $F:\T\to\T$ is a continuous map, $\Delta$ is a positive constant and $h:\T\times\T\to\R$ is the coupling map. If we define the empirical measure as
\begin{equation}\label{eq:def-empirical}
    \mu_N(t):=\frac{1}{N}\sum_{j=1}^N \delta_{x_j(t)},
\end{equation}
and for $\mu\in\mathcal{P}(\T)$ we define
\begin{equation}\label{eq:def-phimu}
    \phi_{\mu}(x):=x+\Delta\int_{\T}h(x,y)\dd\mu(y),\qquad F_{\mu}:=F\circ\phi_{\mu},
\end{equation}
then \eqref{eq:particle-system} can be rewritten as
\begin{equation}\label{eq2:particle-system}
x_i(t+1)=F_{\mu_N(t)}\left(x_i(t)\right),\qquad i=1,\dots,N.
\end{equation}
From \eqref{eq2:particle-system} we obtain
\begin{equation}\label{eq:one-step-empirical}
    \mu_N(t+1)=\left(F_{\mu_N(t)}\right)_{\ast}(\mu_N(t)),
\end{equation}
where for a measurable map $G:\T\to\T$, we let $G_{\ast}:\mathcal{P}(\T)\to\mathcal{P}(\T)$ denote the push-forward map for measures induced by $G$. Thus, if we assume that $\mu_N(t)\to\bar{\rho}_t$ as $N\to\infty$, then we can formally derive the limiting equation of the system:
\begin{equation}\label{eq:mf-dynamics}
\bar{\rho}_{t+1}=\left(F_{\bar{\rho}_t}\right)_{\ast}\bar{\rho}_t.
\end{equation}\par
The equation \eqref{eq:mf-dynamics} plays the role of the mean-field equation in the classical continuous-time mean-field limit. So it is natural to ask whether the closeness of the empirical measure $\mu_N(t)$ and the limiting distribution $\bar{\rho}_t$ can be propagated over time. In this article, we establish results of this type for our system \eqref{eq:particle-system} and \eqref{eq:mf-dynamics}. First we control the growth of the $1$-Wasserstein distance $W_1(\mu_N(t),\bar{\rho}_t)$ to give a Dobrushin-type estimate. Then we derive an exponential growth for the relative entropy of the particle distribution with respect to the tensorized distribution, which implies the propagation of chaos. Finally we consider the first order fluctuation behavior of $\mu_N$ around $\bar{\rho}$. We study the asymptotic behavior of the fluctuation processes $\eta_t^N:=\sqrt{N}(\mu_N(t)-\bar{\rho}_t)$ and prove that $\eta^N$ converges in distribution to some Gaussian process $\eta$ which is a solution to the equation
\begin{equation}\label{eq:etat}
    \eta_{t+1}=\left(F_{\bar{\rho}_t}\right)_{\ast}\eta_t-\Delta F_{\ast}\left(\left((\phi_{\bar{\rho}_t})_{\ast}(\bar{\rho}_t\cdot h\ast\eta_t)\right)'\right).
\end{equation}\par
Particle systems of the form \eqref{eq:particle-system} or its variants are called globally coupled maps (GCMs), which arose as high-dimensional models of complex systems having simple equations but exhibiting various behaviors\cite{Kan89}. The operator acting on $\bar{\rho}_t$ on the right hand side of \eqref{eq:mf-dynamics}, i.e. $\mu\mapsto (F_{\mu})_{\ast}\mu$, is called the self-consistent transfer operator (STO), which was first introduced in the setting of globally coupled maps in \cite{Kan92}.\par
In the continuous-time case, people have established various types of results including the Dobrushin's estimate and Kac's chaos. One can see \cite{Gol16} for these classical results on Lipschitz kernels. Jabin and Wang\cite{JW18} came up with the relative entropy method to show the propagation of chaos for $W^{-1,\infty}$ kernels.  Besides, Wang, Zhao and Zhu\cite{WZZ23} studied the asymptotic behavior of the fluctuation processes $\eta^N$ and showed that $\eta^N$ converges in distribution to a generalized Ornstein-Uhlenbeck process.\par
\subsection{Assumptions}
To state our main results, we list our assumptions as follows:\\
\textbf{(A1)-Regularities of the maps.}\ \ Since we need different regularities of $F$ and $h$ for different results, this assumption is split into three cases:
\begin{enumerate}[label=\textbf{(A1.\arabic*)},left=1em]
    \item The map $F:\T\to\T$ is Lipschitz and $h\in C^1(\T\times\T)$;
    \item The map $F:\T\to\T$ is continuous and $h\in C^3(\T\times\T)$;
    \item The map $F:\T\to\T$ is $C^3$ and $h\in C^3(\T\times\T)$. Furthermore, there exists a constant $c>0$, such that $|F'(x)|\geq c$ for any $x\in\T$.
\end{enumerate}
\textbf{(A2)-Regularity of the solution $\bar{\rho}_t$.}\ \ Let $\bar{\rho}_t,\,t=0,1,\dots,T$ be a strictly positive $C^2$ solution to \eqref{eq:mf-dynamics}, which means there exist constants $\kappa,K>0$, such that
$$\bar{\rho}_t\geq\kappa,\quad\|\bar{\rho}_t\|_{C^2}\leq K,\quad t=0,1,\dots,T.$$
\textbf{(A3)-Independent and identically distributed initial data.}\ \ The initial data $x_i(0),\ i\geq 1$ are i.i.d. random variables on $\T$, with common distribution $\bar{\rho}_0$. It is well-known that under this i.i.d assumption, there exists a centered Gaussian field $\eta_0$, which is a random element on the space of tempered distributions $\mathcal{S}'(\T)$, such that the sequence $\{\eta_0^N\}_{N=1}^{\infty}$ converges in distribution to $\eta_0$ in $\mathcal{S}'(\T)$.\par
We make a remark here on the centered Gaussian field $\eta_0$ in the assumption \textbf{(A3)}. Actually, for any test function $\varphi\in C^{\infty}(\T)$, by the standard Central Limit Theorem, we have
$$\left\langle\varphi,\eta_0^N\right\rangle=\frac{1}{\sqrt{N}}\sum_{i=1}^{N}\left[\varphi(x_i(0))-\langle\varphi,\bar{\rho}_0\rangle\right]\to\mathcal{N}\left(0,\langle\varphi^2,\bar{\rho}_0\rangle-\langle\varphi,\bar{\rho}_0\rangle^2\right)$$
in distribution as $N\to\infty$, where $\mathcal{N}(0,a)$ denotes the centered Gaussian distribution on $\R$ with variance $a$. Hereafter we use the bracket $\langle\cdot,\cdot\rangle$ as a shorthand notation for integration. Thus the Gaussian field $\eta_0$ can be characterized by
\begin{equation}\label{eq:Gaussian}
    \langle\varphi,\eta_0\rangle\sim\mathcal{N}\left(0,\langle\varphi^2,\bar{\rho}_0\rangle-\langle\varphi,\bar{\rho}_0\rangle^2\right),\ \varphi\in C^{\infty}(\T).
\end{equation}
\subsection{Main results}
In this subsection we give the full statements of the main results.\par
Our first main result is a Dobrushin-type estimate, which shows that the closeness of the empirical measure $\mu_N$ and the limiting distribution $\bar{\rho}$ can be propagated over time in the sense of $1$-Wasserstein distance.
\begin{theorem}\label{thm:Dobrushin}
    Under the assumption \textup{\textbf{(A1.1)}}, let $\bar{\rho}_t,t=0,1,\dots$ be a solution to \eqref{eq:mf-dynamics}. Then there exists some constant $C>0$, depending on $\Delta,\mathrm{Lip}(F)$ and $\Vert h\Vert_{C^1}$, such that
    \begin{equation}\label{eq:Dobrushin}
        W_1(\mu_N(t),\bar{\rho}_t)\leq\mathrm{e}^{Ct}W_1(\mu_N(0),\bar{\rho}_0),\qquad t=0,1,\dots,
    \end{equation}
    where $W_1$ denotes the $1$-Wasserstein distance.
\end{theorem}
A brief introduction to Wasserstein distances can be found in Section~\ref{sub:Wasserstein-distances}.\par
The Dobrushin-type estimate is one formulation of propagation of chaos, while another one is the Kac's chaos, which considers the weak convergence of the particle distribution $\rho_{N,t}$ of $X(t)=(x_1(t),\dots,x_N(t))$ to the tensor product $\bar{\rho}_t^{\otimes N}$. The notion of Kac's chaos was first introduced by Kac\cite{Kac56}. Sznitman\cite{Szn91} gave one of the most influential probabilistic treatments for this notion and also popularized the coupling method in the direction of mean-field limit. Jabin and Wang\cite{JW18} came up with the relative entropy method and used this method to prove the Kac's chaos for particle systems with singular kernels. We will also use the relative entropy method in this paper.
\begin{definition}\label{def:relative-entropy}
    Let $E$ be a Polish space and $\mu,\nu$ be two probability measures on $E$. The relative entropy between $\mu$ and $\nu$ is defined by
    $$H(\mu|\nu):=\begin{cases}
        \int_E\frac{\dd\mu}{\dd\nu}\log{\frac{\dd\mu}{\dd\nu}}\dd\nu,\quad&\textup{if }\mu\ll\nu,\\
        \infty,&\textup{otherwise},
    \end{cases}$$
    where $\frac{\dd\mu}{\dd\nu}$ is the Radon-Nikodym derivative of $\mu$ with respect to $\nu$.
\end{definition}
Let $\rho_{N,t}$ be the distribution of $X(t)=(x_1(t),\dots,x_N(t))$ and $\bar{\rho}_t$ be a solution to \eqref{eq:mf-dynamics}. Since both $\rho_{N,t}$ and $\bar{\rho}_t^{\otimes N}$ are probability measures on the $N$-dimensional torus $\T^N$, we will consider a rescaled relative entropy in this paper, which is defined by
\begin{equation}\label{eq:def-HN}
\HN(t)
:=\frac{1}{N} H\left(\rho_{N,t}\big|\bar{\rho}_t^{\otimes N}\right)
=\frac{1}{N}\int_{\T^N}\rho_{N,t}(X)
\log\frac{\rho_{N,t}(X)}{\bar{\rho}_t^{\otimes N}(X)}\dd X.
\end{equation}
Our second main result derives the exponential growth of $\HN(t)$.
\begin{theorem}\label{thm:main}
    Fix $T\in\N$. Under the assumptions \textup{\textbf{(A1.2)}} and \textup{\textbf{(A2)}}, there exists $\Delta_0>0$, such that if $\Delta<\Delta_0$, then it holds that
    \begin{equation}\label{eq:entropy-estimate}
        \HN(t)\leq\mathrm{e}^{Ct}\left(\HN(0)+\frac{1}{N}\right),\quad t=0,1,\dots,T,
    \end{equation}
    where $\Delta_0$ and $C$ depend on $\| h\|_{C^3},\kappa$ and $K$.
\end{theorem}
Now we give the rigorous definition for the Kac's chaos and claim that Theorem~\ref{thm:main} implies the Kac's chaos.
\begin{definition}\label{def:Kacs-chaos}
     Let $E$ be a Polish space and $\rho$ be a probability measure on $E$. A sequence $P_N$ of symmetric $N$-particle probability measures on $E^N$ for all $N\geq 1$ is said to be \textit{$\rho$-chaotic} if for all $k\geq 1$, $P_{N:k}$ converges weakly to $\rho^{\otimes k}$ as $N\to\infty$, where $P_{N:k}$ is the $k$-marginal of $P_N$.
\end{definition}
\begin{corollary}\label{cor:Kacs-chaos}
    Fix $T\in\N$. Under the assumptions \textup{\textbf{(A1.2)}}, \textup{\textbf{(A2)}} and \textup{\textbf{(A3)}}, there exists $\Delta_0>0$, such that if $\Delta<\Delta_0$, then $\rho_{N,t}$ is $\bar{\rho}_t$-chaotic for $t=0,1,\dots,T$, where $\Delta_0$ and $C$ depends on $\| h\|_{C^3},\kappa$ and $K$.
\end{corollary}
After the results above on propagation of chaos, we want to study the behaviors of higher-order errors while using the empirical measure $\mu_N(t)$ to approximate the limiting distribution $\bar{\rho}_t$. So it is natural to consider the fluctuation processes
\begin{equation}\label{eq:def-etatN}
    \eta_t^N:=\sqrt{N}(\mu_N(t)-\bar{\rho}_t).
\end{equation}
For the first-order interacting particle system, Wang, Zhao and Zhu\cite{WZZ23} applied the relative entropy method to prove that $\eta^N$ converges in distribution to a generalized Ornstein-Uhlenbeck process. Our last result shows that for our system, the process $\eta^N=(\eta_t^N)_{0\leq t\leq T}$ converges in distribution to some Gaussian process.
\begin{theorem}\label{thm:CLT}
    Fix $T\in\N$. Under the assumptions \textup{\textbf{(A1.3)}}, \textup{\textbf{(A2)}} and \textup{\textbf{(A3)}}, assume further $\eta_0$ in the assumption \textup{\textbf{(A3)}} is characterized by
    $$\langle\varphi,\eta_0\rangle\sim\mathcal{N}\left(0,\langle\varphi^2,\bar{\rho}_0\rangle-\langle\varphi,\bar{\rho}_0\rangle^2\right),\quad\varphi\in C^{\infty}(\T),$$
    and let $\eta=(\eta_t)_{0\leq t\leq T}$ be the solution to \eqref{eq:etat} with initial datum $\eta_0$. If $\Delta<\Delta_0$, where $\Delta_0$ is the constant in Theorem~\ref{thm:main}, then the sequence $\eta_N=(\eta_t^N)_{0\leq t\leq T}$ converges in distribution to $\eta$ in the space $(H^{-\alpha}(\T))^{T+1}$ as $N\to\infty$ for every $\alpha>1/2$. Furthermore, the process $\eta$ is a Gaussian field satisfying
    $$\langle\varphi,\eta_t\rangle\sim\mathcal{N}\left(0,\langle|Q_{0,t}\varphi|^2,\bar{\rho}_0\rangle-\langle Q_{0,t}\varphi,\bar{\rho}_0\rangle^2\right),$$
    for each test function $\varphi$ and $0\leq t\leq T$. Here the time evolution operator $\{Q_{0,t}\}_{0\leq t\leq T}$ is given by $Q_{0,t}:=Q_{0,1}\circ\cdots\circ Q_{t-1,t},\ t\geq 0$, with
\begin{equation}\label{eq:def-Qtt+1}
    (Q_{t,t+1}\varphi)(y):=(\varphi\circ F_{\bar{\rho}_t})(y)+\Delta\int_{\T}(\varphi\circ F)'(\phi_{\bar{\rho}_t}(x))\bar{\rho}_t(x)h(x,y)\dd x,\quad y\in\T,\ t\geq 0.
\end{equation}
\end{theorem}
Theorem 1.6 is a qualitative result which states the convergence without giving any rates. Actually, we can also obtain the following Berry-Esseen/Stein type quantitative result, claiming that the convergence is of order $O(N^{-1/2})$.
\begin{theorem}\label{thm:quantitative}
    Fix $T\in\N$ and $m\in\N$. Assume $\Delta<\Delta_0$, where $\Delta_0$ is the constant in Theorem~\ref{thm:main}. Under the assumptions \textup{\textbf{(A1.3)}}, \textup{\textbf{(A2)}} and \textup{\textbf{(A3)}}, given any $m$ test functions $\varphi_1,\dots,\varphi_m\in C^2(\T)$ and a function $\Psi\in C_b^3(\R^m)$, we have
    \begin{equation}\label{eq:quantitative}
        \left|\mathbb{E}\Psi\left(\langle\varphi_1,\eta_t^N\rangle,\dots,\langle\varphi_m,\eta_t^N\rangle\right)-\mathbb{E}\Psi\left(\langle\varphi_1,\eta_t\rangle,\dots,\langle\varphi_m,\eta_t\rangle\right)\right|\leq\frac{C\mathrm{e}^{Ct}}{\sqrt{N}},\quad t=0,\dots,T,
    \end{equation}
    where the constant $C>0$ depends on $\Delta,m,\|\Psi\|_{C_b^3},\|F\|_{C^2},\|h\|_{C^2}$ and $\|\varphi_i\|_{C^2}(1\leq i\leq m)$.
\end{theorem}
Our main results study the mean-field coupled maps from the perspective of propagation of chaos and Gaussian fluctuations. Results of this type are quite well-known in the theory of mean field limit. To the best of the author's knowledge, however, such results are established here for globally coupled maps for the first time.
\subsection{Related literature}
One of the first works on GCMs was \cite{Kan89}, where they studied the globally coupled system
\begin{equation}
    x_i(t+1)=(1-\eps)f(x_i(t))+\frac{\eps}{N}\sum_{j=1}^{N}f(x_j(t)),\quad i=1,\dots,N,
\end{equation}
with $x_i(t)\in[-1,1]$ and $\eps\geq 0$. And the uncoupled map $f:[-1,1]\to[-1,1]$ is chosen from the logistic family
$$f(x):=f_a(x)=1-ax^2$$
with $0\leq a\leq 2$. By simulating the dynamics, they observed that when the parameters $\eps,a$ and the initial data vary, the behaviors of the particles can change drastically, from synchronization to complete chaos.\par
Since then, many aspects of GCMs have been studied. To the author's knowledge, the closest works to ours are \cite{EP95} and \cite{EP97}, where they studied the evolution of the mean-field coupling term
$$h_N(t):=\frac{1}{N}\sum_{j=1}^{N}f(x_j(t))$$
for different maps $f$ and gave quantitative fluctuation rates for $h_N(t)$. Other works include studies of Lyapunov exponents\cite{Kan95,Mor97}, ergodic theory\cite{SB16,Kel97} of GCMs and so on. Recently, Ha, Lee and Yoon\cite{HLY25} also studied a time-discrete infinite Kuramoto model of similar type.\par
The notion of STOs was first introduced in the setting of GCMs in \cite{Kan92}. People have used different approaches to study the stability of fixed points of STOs in different settings, including the functional analytic approach\cite{BKST18,Kel00} and the cone approach\cite{Liv95,Tan22}. The linear response to perturbations for STOs has also been widely studied\cite{ST21}.\par
The study of propagation of chaos can be traced back to Kac\cite{Kac56}, who introduced the concept of Kac's chaos. McKean\cite{McK66,McK67} connected propagation of chaos with nonlinear Markov processes. Dobrushin\cite{Dob79} used Wasserstein distances to give another formulation for propagation of chaos. Sznitman\cite{Szn91} gave one of the most influential probabilistic frameworks for propagation of chaos. Mischler and Mouhot\cite{MM13}, Hauray and Mischler\cite{HM14}, Deng, Hani and Ma\cite{DHM24} studied the propagation of chaos for the Boltzmann equation, which is one of the deepest areas since the Boltzmann equation is nonlinear and collision-driven.\par
A major modern direction is propagation of chaos for singular forces, including Coulomb, Newtonian, vortex, and hard-sphere interactions. These are substantially harder because classical Lipschitz estimates fail. Pickl\cite{BP16,LP17} came up with the particle trajectory method, which shows that the trajectories of the original particle system and the limiting mean-field system are close in probability. Jabin and Wang\cite{JW18} proposed the relative entropy method, which controls the growth of the relative entropy between the particle distribution and the tensorized distribution. Serfaty\cite{Ser20} introduced the modulated energy method, which controls the growth of modulated energy associated to the empirical measure and the limiting distribution. All the methods above give rise to the propagation of chaos in some sense.\par
One of the earliest works on Gaussian fluctuations is due to It\^{o}\cite{Ito83}, where he showed that for the system of 1D independent and identically distributed Brownian motions, the limit of the corresponding fluctuations is a Gaussian process. The fluctuation in path space type result was first obtained by Tanaka and Hitsuda\cite{TH81} and by Tanaka\cite{Tan84} for interacting diffusions. Fernandez and M\'{e}l\'{e}ard\cite{FM97} proposed a Hilbertian approach to study interacting diffusions with sufficiently regular coefficients. Wang, Zhao and Zhu\cite{WZZ23} used Jabin and Wang's relative entropy method to study the Gaussian fluctuations for interacting particle systems with singular kernels, where they proved that the fluctuation processes will converge in distribution to a generalized Ornstein-Uhlenbeck process. Recently, Hao, Zhang and Zhao\cite{HZZ26} gave a quantitative estimate of the fluctuations for kinetic McKean--Vlasov SDEs with singular interaction kernels.
\subsection{Structure of the paper}
The rest of the paper will be organized as follows. In Section~\ref{sec:preliminaries} we will give some preliminaries for our results, including brief introductions to Wasserstein distances, push-forward maps, the data processing inequality and Sobolev spaces on the torus. The proof of the Dobrushin-type estimate (Theorem~\ref{thm:Dobrushin}) will be given in Section~\ref{sec:Dobrushin-type-estimate}, which is a rather simple one.\par
Section~\ref{sec:relative-entropy-bound} will be devoted to proving the relative entropy bound(Theorem~\ref{thm:main}). First in Section~\ref{sub:large-deviation} we will give some large deviation type estimates, which have been established by Jabin and Wang in \cite{JW18}. Then in Section~\ref{sub:one-step} we propose a one-step inequality that controls the difference $\HN(t+1)-\HN(t)$. Next in Section~\ref{sub:estimates-on-Rt} we further estimate the difference term and use the estimate to complete the proof of Theorem~\ref{thm:main}. Finally in Section~\ref{sub:propagation-of-chaos} we give a short proof for Corollary~\ref{cor:Kacs-chaos}, which is a standard one.\par
Section~\ref{sec:Gaussian-fluctuations} is concerned with the central limit theory(Theorem~\ref{thm:CLT}) for our system. First in Section~\ref{sub:formal-derivation} we propose the formal derivation of the equation \eqref{eq:etat} satisfied by $\eta_t$. Then in Section~\ref{sub:time-evolution} we discuss the time evolution operator $Q_{0,t}$ and the reason why $\eta$ is a Gaussian process. Next in Section~\ref{sub:convergence} we use the relative entropy method to complete the proof of Theorem~\ref{thm:CLT}. Finally in Section~\ref{sub:quantitative} we establish the quantitative estimate(Theorem~\ref{thm:quantitative}).
\section{Preliminaries}\label{sec:preliminaries}
Throughout the rest of the paper, the letter $C$ denotes a generic positive constant whose value may vary from line to line. Its dependence will be specified in the statements when needed.
\subsection{Wasserstein Distances}\label{sub:Wasserstein-distances}
Let $E$ be a Polish space with distance $d$ and $\mathcal{P}(E)$ be the set of all Borel probability measures on $E$. For $p\geq 1$, we let $\mathcal{P}_p(E)$ denote the set of all Borel probability measures with a finite moment of order $p$, i.e.
$$\mathcal{P}_p(E):=\left\{\mu\in\mathcal{P}(E)\Big|\int_E d(x,x_0)^p\dd\mu(x)<+\infty\right\},$$
where $x_0\in E$ is a fixed point. One can easily check that the definition of $\mathcal{P}_p(E)$ does not depend on the choice of $x_0$.
\begin{definition}\label{def:Wasserstein}
    Let $\mu,\nu\in\mathcal{P}_p(E)$. A \textit{coupling} of $\mu$ and $\nu$ is a probability measure $\pi$ on $E\times E$ such that
    $$\pi(A\times E)=\mu(A),\qquad \pi(E\times B)=\nu(B),$$
    for all Borel sets $A,B\subseteq E$. The set of all couplings of $\mu$ and $\nu$ is denoted by $\Pi(\mu,\nu)$. The \textit{$p$-Wasserstein distance} of $\mu$ and $\nu$ is defined by
    \begin{equation}\label{eq:Wasserstein}
        W_p(\mu,\nu):=\left(\inf_{\pi\in\Pi(\mu,\nu)}\int_{E\times E}d(x,y)^p\dd\pi(x,y)\right)^{\frac{1}{p}}.
    \end{equation}
\end{definition}
\begin{remark}\label{rmk:Wasserstein}
    Equipped with the distance function $W_p$, the space $\mathcal{P}_p(E)$ becomes a Polish space.
\end{remark}
For the $1$-Wasserstein distance, we have the following Kantorovich-Rubinstein dual formula:
\begin{proposition}\label{prop:1-Wasserstein}
    Let $\mu,\nu\in\mathcal{P}_1(E)$. Then
    \begin{equation}\label{eq:1-Wasserstein}
        W_1(\mu,\nu)=\sup_{\substack{\varphi\in\mathrm{Lip}(E)\\\mathrm{Lip}(\varphi)\leq 1}}\left|\int_E\varphi(x)\dd(\mu-\nu)(x)\right|,
    \end{equation}
    where $\mathrm{Lip}(E)$ denotes the set of all Lipschitz functions on $E$, and for $\varphi\in\mathrm{Lip}(E)$, we denote the Lipschitz constant of $\varphi$ by $\mathrm{Lip}(\varphi)$.
\end{proposition}
Besides, we can characterize the topology generated by the $1$-Wasserstein distance by the following proposition.
\begin{proposition}\label{prop:topology}
    Let $\mu_n,\mu\in\mathcal{P}_1(E)$ and fix an $x_0\in E$. Then $W_1(\mu_n,\mu)\to 0$ as $n\to\infty$ if and only if:\\
    \textup{(i)} $\mu_n$ converges weakly to $\mu$ in $\mathcal{P}(E)$ as $n\to \infty$;\\
    \textup{(ii)} $\lim_{n\to\infty}\int_E d(x,x_0)\dd\mu_n=\int_E d(x,x_0)\dd\mu$.\\
\end{proposition}
\begin{remark}\label{rmk:topology}
    When $E$ is compact, the distance function $d(\cdot,x_0)$ is a bounded continuous function on $E$, and hence (i) implies (ii). So in this case, the $1$-Wasserstein distance generates the weak topology on $\mathcal{P}_1(E)$.
\end{remark}
\subsection{Push-forward maps}\label{sub:push-forward-maps}
As introduced in Section~\ref{sec:introduction}, the push-forward maps play an important role in our system. So it is necessary to write down the explicit expression for push-forward maps on the torus.\par
For a measurable map $f:\T^d\to\T^d$, the push-forward map $f_{\ast}:\mathcal{P}(\T^d)\to\mathcal{P}(\T^d)$ induced by $f$ is characterized by
\begin{equation}\label{eq:def-push-forward}
    \int_{\T^d}\varphi(x)\dd(f_{\ast}\mu)(x)=\int_{\T^d}\varphi(f(x))\dd\mu(x),
\end{equation}
for all measurable functions $\varphi$ on $\T^d$. For a $\mu\in\mathcal{P}(\T^d)$ that is absolutely continuous with respect to the Lebesgue measure on $\T^d$, we abuse notation to denote both the probability measure and the density function by $\mu$. Under certain circumstances, e.g. $f$ is a diffeomorphism on $\T^d$, $f_{\ast}$ maps absolutely continuous measures to absolutely continuous measures, in which case we can write down the explicit expression for $f_{\ast}\mu$ by an application of the change of variables formula.
\begin{proposition}\label{prop:general-push-forward}
    Let $f:\T^d\to\T^d$ be a $C^1$ diffeomorphism and $\mu\in L^1(\T^d)\cap\mathcal{P}(\T^d)$ be an absolutely continuous Borel probability measure on $\T^d$. Then we have the formula
    \begin{equation}
        (f_{\ast}\mu)(x)=\frac{\mu(f^{-1}(x))}{|Df|(f^{-1}(x))},\quad x\in\T^d,
    \end{equation}
    where $|Df|$ denotes the absolute value of the Jacobian of $f$.
\end{proposition}
We let $\Phi:\T^N\to\T^N$ denote the one-step transformation in \eqref{eq:particle-system}, i.e. $\Phi=(\Phi_1,\Phi_2,\dots,\Phi_N)$ with
\begin{equation}\label{eq:def-Phi}
    \Phi_i(x_1,\dots,x_N)=x_i+\frac{\Delta}{N}\sum_{j=1}^{N}h(x_i,x_j),\quad i=1,\dots,N.
\end{equation}
From \eqref{eq:def-phimu} and \eqref{eq:def-Phi}, it is obvious that for $\Delta$ small enough, both $\Phi$ and $\phi_{\mu}$ are diffeomorphisms, in which case we can apply Proposition~\ref{prop:general-push-forward} to write down the explicit expressions for $\Phi_{\ast}$ and $(\phi_{\mu})_{\ast}$.
\begin{corollary}\label{cor:push-forward}
    There exists a constant $\Delta_1>0$, depending only on $\|h\|_{C^1}$, such that when $\Delta<\Delta_1$, the mapping $\Phi$ becomes a diffeomorphism on $\T^N$ and $\phi_{\mu}$ also becomes a diffeomorphism on $\T$ for any $\mu\in\mathcal{P}(\T)$. In this case, for any $\rho\in L^1(\T^N)\cap\mathcal{P}(\T^N)$ and $\bar{\rho}\in L^1(\T)\cap\mathcal{P}(\T)$, we have
    \begin{equation}\label{eq:push-forward}
        (\Phi_{\ast}\rho)(X)=\frac{\rho(\Phi^{-1}(X))}{|D\Phi|(\Phi^{-1}(X))},\quad((\phi_{\mu})_{\ast}\bar{\rho})(x)=\frac{\bar{\rho}(\phi_{\mu}^{-1}(x))}{\phi_{\mu}'(\phi_{\mu}^{-1}(x))}.
    \end{equation}
\end{corollary}
\subsection{The data processing inequality}\label{sub:the-data-processing-inequality}
The data processing inequality is a significant inequality in the field of information theory, which states that the relative entropy of two probability measures does not increase under any push-forward map.
\begin{proposition}\label{prop:data-processing}
    Let $E$ be a Polish space and $\mu,\nu\in\mathcal{P}(E)$ be two Borel probability measures on $E$. Suppose that $f:E\to E$ is a continuous map. Then we have
    \begin{equation}\label{eq:data-processing}
        H(f_{\ast}\mu|f_{\ast}\nu)\leq H(\mu|\nu).
    \end{equation}
\end{proposition}
\begin{remark}\label{rmk:data-processing}
    Actually, we can obtain the explicit expression for the loss term when applying the data processing inequality, using the regular condition probability distribution(r.c.p.d.). Indeed, if $X,Y$ are two Polish spaces and $f:X\to Y$ is a continuous map, let $\mu,\nu\in\mathcal{P}(X)$ be two Borel probability measures on $X$ and $\{\mu^y|y\in Y\},\ \{\nu^y|y\in Y\}$ be the r.c.p.d. of $\mu$ and $\nu$ given $f$ respectively, then we have
    \begin{equation}\label{eq:data-processing-loss}
        H(\mu|\nu)=H(f_{\ast}\mu|f_{\ast}\nu)+\int_Y H(\mu^y|\nu^y)\dd(f_{\ast}\mu)(y).
    \end{equation}
\end{remark}
\subsection{Sobolev spaces on the torus}\label{sub:Sobolev-spaces-on-the-torus}
We use $\{e_k\}_{k\in\mathbb{Z}}$ to represent the Fourier basis on $\T$ or $e_k(x)=\mathrm{e}^{2\pi\mathrm{i}kx}$ for $k\in\mathbb{Z}$. For simplicity, we define $\langle k\rangle:=\sqrt{1+k^2}$. The norm of the Sobolev space $H^{\alpha}(\T)$, $\alpha\in\R$, is defined by
\begin{equation}\label{eq:def-Halpha}
    \|f\|_{H^{\alpha}}^2:=\sum_{k\in\mathbb{Z}}\langle k\rangle^{2\alpha}|\langle f,e_k\rangle|^2.
\end{equation}
With the definition \eqref{eq:def-Halpha}, one can easily check the dual formula
\begin{equation}\label{eq:dual}
     \|f\|_{H^{-\alpha}}=\sup_{\|\varphi\|_{H^{\alpha}}\leq 1}|\langle\varphi,f\rangle|,
\end{equation}
for any $\alpha\in\R$ and $f\in H^{-\alpha}(\T)$.\par
When $\alpha\in\mathbb{N}$, it is well-known that an equivalent norm for $H^{\alpha}$ is $\|f\|:=(\sum_{j=0}^{\alpha}\|f^{(j)}\|_{L^2}^2)^{\frac{1}{2}}$.
\begin{proposition}\label{prop:Hn}
    Let $n\in\mathbb{N}$. Then there exists a constant $C\geq 1$, such that
    \begin{equation}\label{eq:Hn}
        C^{-1}\|f\|_{H^n}^2\leq\sum_{j=0}^{n}\|f^{(j)}\|_{L^2}^2\leq C\|f\|_{H^n}^2,\quad\forall f\in H^n(\T).
    \end{equation}
\end{proposition}
\section{Dobrushin-type Estimate}\label{sec:Dobrushin-type-estimate}
In this section we give the proof of Theorem~\ref{thm:Dobrushin}, which is rather simple. The estimate \eqref{eq:Dobrushin} follows by an application of the Kantorovich-Rubinstein dual formula(Proposition~\ref{prop:1-Wasserstein}) for the $1$-Wasserstein distance and some direct calculations.
\begin{proof}[Proof of Theorem~\ref{thm:Dobrushin}]
    We will use Proposition ~\ref{prop:1-Wasserstein} to compute the $1$-Wasserstein distance of two measures.\par
    Fix any Lipschitz function $\varphi$ on $\T$ with $\mathrm{Lip}(\varphi)\leq 1$, then we have
    \begin{align*}
        &\langle\varphi,\mu_N(t+1)-\bar{\rho}_{t+1}\rangle\\
        =&\langle\varphi,(F_{\mu_N(t)})_{\ast}\mu_N(t)-(F_{\bar{\rho}_t})_{\ast}\bar{\rho}_t\rangle\\
        =&\langle\varphi,(F_{\mu_N(t)})_{\ast}\mu_N(t)-(F_{\mu_N(t)})_{\ast}\bar{\rho}_t\rangle+\langle\varphi,(F_{\mu_N(t)})_{\ast}\bar{\rho}_t-(F_{\bar{\rho}_t})_{\ast}\bar{\rho}_t\rangle\\
        =&\langle\varphi\circ F\circ\phi_{\mu_N(t)},\mu_N(t)-\bar{\rho}_t\rangle+\langle\varphi\circ F\circ\phi_{\mu_N(t)}-\varphi\circ F\circ\phi_{\bar{\rho}_t},\bar{\rho}_t\rangle\\
        =&\I+\II.
    \end{align*}\par
    For $\I$, since
    $$\phi_{\mu_N(t)}'(x)=1+\Delta\int_{\T}\partial_1 h(x,y)\dd(\mu_N(t))(y),$$
    we have $\mathrm{Lip}(\phi_{\mu_N(t)})\leq 1+\Delta\Vert h\Vert_{C^1}$. And hence
    \begin{align*}
        |\I|&\leq\mathrm{Lip}(\varphi\circ F\circ\phi_{\mu_N(t)})\cdot W_1(\mu_N(t),\bar{\rho}_t)\\
        &\leq\mathrm{Lip}(\varphi)\cdot\mathrm{Lip}(F)\cdot\mathrm{Lip}(\phi_{\mu_N(t)})\cdot W_1(\mu_N(t),\bar{\rho}_t)\\
        &\leq\mathrm{Lip}(F)(1+\Delta\Vert h\Vert_{C^1})\cdot W_1(\mu_N(t),\bar{\rho}_t).
    \end{align*}\par
    For $\II$, since
    $$\phi_{\mu_N(t)}(x)-\phi_{\bar{\rho}_t}(x)=\Delta\int_{\T}h(x,y)\dd(\mu_N(t)-\bar{\rho}_t)(y),$$
    we have $\Vert\phi_{\mu_N(t)}-\phi_{\bar{\rho}_t}\Vert_{L^{\infty}}\leq\Delta\Vert h\Vert_{C^1}\cdot W_1(\mu_N(t),\bar{\rho}_t)$. And thus
    \begin{align*}
        |\II|&\leq\Vert\varphi\circ F\circ\phi_{\mu_N(t)}-\varphi\circ F\circ\phi_{\bar{\rho}_t}\Vert_{L^{\infty}}\\
        &\leq\mathrm{Lip}(\varphi\circ F)\cdot\Vert\phi_{\mu_N(t)}-\phi_{\bar{\rho}_t}\Vert_{L^{\infty}}\\
        &\leq\Delta\mathrm{Lip}(F)\Vert h\Vert_{C^1}\cdot W_1(\mu_N(t),\bar{\rho}_t).
    \end{align*}\par
    Combining the estimates above, we obtain
    $$|\langle\varphi,\mu_N(t+1)-\bar{\rho}_{t+1}\rangle|\leq\mathrm{Lip}(F)(1+2\Delta\Vert h\Vert_{C^1})\cdot W_1(\mu_N(t),\bar{\rho}_t).$$
    Taking supremum over all Lipschitz functions $\varphi$ on $\T$ with $\mathrm{Lip}(\varphi)\leq 1$, by Proposition~\ref{prop:1-Wasserstein} again we get
    $$W_1(\mu_N(t+1),\bar{\rho}_{t+1})\leq\mathrm{Lip}(F)(1+2\Delta\Vert h\Vert_{C^1})\cdot W_1(\mu_N(t),\bar{\rho}_t).$$
    Iterating this inequality for $t$ times gives the desired estimate \eqref{eq:Dobrushin}.
\end{proof}
\begin{remark}\label{rmk:Dobrushin}
    Recall that the Dobrushin-type estimate \eqref{eq:Dobrushin} says that $W_1(\mu_N(t),\bar{\rho}_t)\leq\mathrm{e}^{Ct}W_1(\mu_N(0),\bar{\rho}_0)$, which yields that if we choose the initial data $X_N(0)$ appropriately such that $W_1(\mu_N(0),\bar{\rho}_0)\to 0$ as $N\to\infty$, then $W_1(\mu_N(t),\bar{\rho}_t)\to 0$ as $N\to\infty$ for any $t\geq 0$. This is what we exactly mean when we say in Section~\ref{sec:introduction} that the closeness of the empirical measure $\mu_N$ and the limiting distribution $\bar{\rho}$ can be propagated along with the time in the sense of $1$-Wasserstein distance.
\end{remark}
\section{Relative Entropy Bound and Propagation of Chaos}\label{sec:relative-entropy-bound}
In this section we are going to prove Theorem~\ref{thm:main} and Corollary~\ref{cor:Kacs-chaos}. The idea of proof is as follows. First we use the expression for the push-forward maps and apply the data processing inequality to obtain a one-step inequality $\HN(t+1)\leq \HN(t)+R_t$, where $R_t$ is the integral of some function with respect to the particle density $\rho_{N,t}$. Then we use Taylor's expansion to expand the integrand in $R_t$, and apply the Gibbs' inequality (Proposition~\ref{prop:gibbs}) to change the measure $\rho_{N,t}$ into the tensor product $\bar{\rho}_t^{\otimes N}$. Next we use different large deviation type estimates to bound the first order term and second order terms respectively and obtain that $R_t\leq C\left(\HN(t)+\frac{1}{N}\right)$. Then the estimate \eqref{eq:entropy-estimate} follows by induction. Finally, the Kac's chaos follows by an application of the classical Csisz\'{a}r-Kullback-Pinsker inequality.
\subsection{Large deviation type estimates}\label{sub:large-deviation}
As mentioned before, we need some large deviation type estimates to bound the first order term and second order terms in the Taylor's expansion. The following two results are from Jabin and Wang\cite{JW18}, and we adapt a bit for convenience. They state that under some cancellation rules of the function $\phi$, the expectations of some exponentials related to $\phi$ with respect to the tensorized law $\rho^{\otimes N}$ are uniformly bounded. In \cite{JW18}, Jabin and Wang gave a combinatorial proof for these results. Later, Guo, Liang and Wang\cite{GHW26} used the technique of generating functions to propose a shorter proof.
\begin{proposition}[Jabin and Wang {\cite[Theorem 4]{JW18}}]\label{prop:exp1}
    Let $E$ be a Polish space and $\rho\in\mathcal{P}(E)$ be a probability measure on $E$. Suppose that $\phi:E\times E\to\R$ is a measurable function satisfying the following two cancellation rules:
    \begin{equation}\label{eq:cancellation1}
        \int_E\phi(x,y)\dd\rho(x)=0,\quad\forall y\in E;\qquad
        \int_E\phi(x,y)\dd\rho(y)=0,\quad\forall x\in E.
    \end{equation}
    Then there exists a constant $C_1>0$, such that when $\|\phi\|_{L^{\infty}}\leq C_1$, one has
    \begin{equation}\label{eq:exp1}
        \sup_{N\geq 2}\int_{E^N}   \rho^{\otimes N}(X_N)\exp{\left(\frac{1}{N}\sum_{i,j=1}^{N}\phi(x_i,x_j)\right)}\dd X_N\leq C<+\infty,
    \end{equation}
    where $X_N=(x_1,x_2,\dots,x_N)$ for $N\geq 2$ and $C$ is a generic constant independent of $N$.
\end{proposition}
\begin{proposition}[Jabin and Wang {\cite[Theorem 3]{JW18}}]\label{prop:exp2}
     Let $E$ be a Polish space and $\rho\in\mathcal{P}(E)$ be a probability measure on $E$. Suppose that $\phi:E\times E\to\R$ is a measurable function satisfying the cancellation rule
     \begin{equation}\label{eq:cancellation2}
         \int_E\phi(x,y)\dd\rho(y)=0,\quad\forall x\in E.
     \end{equation}
     Then there exists a constant $C_2>0$, such that when $\|\phi\|_{L^{\infty}}\leq C_2$, one has
     \begin{equation}\label{eq:exp2}
         \sup_{N\geq 2}\int_{E^N}   \rho^{\otimes N}(X_N)\exp{\left(\frac{1}{N}\sum_{j,k=1}^{N}\phi(x_1,x_j)\phi(x_1,x_k)\right)}\dd X_N\leq C<+\infty,
     \end{equation}
     where $X_N=(x_1,x_2,\dots,x_N)$ for $N\geq 2$ and $C$ is a generic constant independent of $N$.
\end{proposition}
In the proof below, Proposition~\ref{prop:exp1} and Proposition~\ref{prop:exp2} will be used to bound the first order term and second order terms respectively.
\subsection{A one-step inequality}\label{sub:one-step}
In the continuous-time case\cite{JW18}, the estimate of type \eqref{eq:entropy-estimate} follows from computing the time derivative of the rescaled relative entropy $\HN(t)$. In the discrete-time case, the corresponding notion is the difference $\HN(t+1)-\HN(t)$. Therefore, we first use the expression for the push-forward maps and apply the data processing inequality to obtain a one-step inequality as follows.
\begin{proposition}\label{prop:one-step}
    Assume $\Delta<\Delta_1$, where $\Delta_1$ is the constant in Corollary~\ref{cor:push-forward}. Then under the assumption of Theorem~\ref{thm:main}, one has
    \begin{equation}\label{eq:one-step}
    \HN(t+1)\leq \HN(t)+R_t,\quad t=0,1,\dots,T-1,
    \end{equation}
    where
    \begin{equation}\label{eq:def-Rt}
        R_t:=\frac{1}{N}\int_{\T^N}\rho_{N,t}(X)\left[\sum_{i=1}^N\log{\frac{\bar{\rho}_t(x_i)}{\bar{\rho}_t(\phi_{\bar{\rho}_t}^{-1}(\Phi_i(X)))}}+\log{\frac{\prod_{i=1}^N \phi_{\bar{\rho}_t}'(\phi_{\bar{\rho}_t}^{-1}(\Phi_i(X)))}{|D\Phi|(X)}}\right]\dd X.
    \end{equation}
\end{proposition}
\begin{proof}
    According to the equations \eqref{eq:particle-system} and \eqref{eq:mf-dynamics}, and applying the data processing inequality(Proposition~\ref{prop:data-processing}), we have
    \begin{align*}
        \HN(t+1)&=\frac{1}{N}\int_{\T^N}\rho_{N,t+1}(X)\log{\frac{\rho_{N,t+1}(X)}{\bar{\rho}_{t+1}^{\otimes N}(X)}}\dd X\\
        &=\frac{1}{N}\int_{\T^N}((F^{\otimes N})_{\ast}\Phi_{\ast}\rho_{N,t})(X)\log{\frac{((F^{\otimes N})_{\ast}\Phi_{\ast}\rho_{N,t})(X)}{((F^{\otimes N})_{\ast}(\phi_{\bar{\rho}_t}^{\otimes N})_{\ast}\bar{\rho}_t^{\otimes N})(X)}}\dd X\\
        &\leq\frac{1}{N}\int_{\T^N}(\Phi_{\ast}\rho_{N,t})(X)\log{\frac{(\Phi_{\ast}\rho_{N,t})(X)}{((\phi_{\bar{\rho}_t}^{\otimes N})_{\ast}\bar{\rho}_t^{\otimes N})(X)}}\dd X\\
        &=\frac{1}{N}\int_{\T^N}\rho_{N,t}(X)\log{\frac{(\Phi_{\ast}\rho_{N,t})(\Phi(X))}{\prod_{i=1}^{N}((\phi_{\bar{\rho}_t})_{\ast}\bar{\rho}_t)(\Phi_i(X))}}\dd X.
    \end{align*}
    Then we can use the explicit expressions for push-forward maps \eqref{eq:push-forward} to continue with
    \begin{align*}
        \HN(t+1)&\leq\frac{1}{N}\int_{\T^N}\rho_{N,t}(X)\log{\frac{\rho_{N,t}(X)/|D\Phi|(X)}{\prod_{i=1}^{N}(\bar{\rho}_t(\phi_{\bar{\rho}_t}^{-1}(\Phi_i(X)))/\phi_{\bar{\rho}_t}'(\phi_{\bar{\rho}_t}^{-1}(\Phi_i(X))))}}\dd X\\
        &=\frac{1}{N}\int_{\T^N}\rho_{N,t}(X)\log{\frac{\rho_{N,t}(X)}{\bar{\rho}_t^{\otimes N}(X)}}\dd X+R_t\\
        &=\HN(t)+R_t.
    \end{align*}
\end{proof}
\subsection{Estimates on $R_t$ and completing the proof of Theorem~\ref{thm:main}}\label{sub:estimates-on-Rt}
Now we need to estimate the term $R_t$. We want to bound $R_t$ by $C\HN(t)$ plus some small terms, and then we can complete the proof by induction. Before that, for the readers' convenience, we give the full statement of the Gibbs' inequality here. It is used to change the measure that expectations are taken with respect to from the particle density $\rho_{N,t}$ to the tensor product $\bar{\rho}_t^{\otimes N}$. For a short proof, one can see \cite[Lemma 1]{JW18}.
\begin{proposition}\label{prop:gibbs}
     Let $E$ be a Polish space and $\mu,\nu\in\mathcal{P}(E)$ be two probability measures on $E$. Suppose that $\Phi:E\to\R$ is a measurable function. Then $\forall\lambda>0$,
     \begin{equation}\label{eq:gibbs}
         \int_E\Phi(x)\dd\mu(x)\leq\frac{1}{\lambda}H(\mu|\nu)+\frac{1}{\lambda}\log{\int_{E}\mathrm{e}^{\lambda\Phi(x)}\dd\nu(x)}.
     \end{equation}
\end{proposition}
Now we can propose the estimates on $R_t$, which is the major part of this section.
\begin{proposition}\label{prop:estimate-Rt}
    Let $R_t$ be as in Proposition~\ref{prop:one-step}. Then under the assumption of Theorem~\ref{thm:main}, there exists a constant $\Delta_0>0$, such that when $\Delta<\Delta_0$, one has the estimate
    \begin{equation}\label{eq:estimate-Rt}
        R_t\leq C\left(\HN(t)+\frac{1}{N}\right),
    \end{equation}
    where $\Delta_0$ and $C$ depend on $\| h\|_{C^3},\kappa$ and $K$.
\end{proposition}
\begin{proof}
    We first rewrite $R_t$ as
    \begin{align*}
        R_t=&\frac{1}{N}\int_{\T^N}\rho_{N,t}(X)\sum_{i=1}^N\log{\frac{\bar{\rho}_t(x_i)}{\bar{\rho}_t(\phi_{\bar{\rho}_t}^{-1}(\Phi_i(X)))}}\dd X\\
        &+\frac{1}{N}\int_{\T^N}\rho_{N,t}(X)\sum_{i=1}^{N}\log{\frac{\phi_{\bar{\rho}_t}'(\phi_{\bar{\rho}_t}^{-1}(\Phi_i(X)))}{\phi_{\bar{\rho}_t}'(x_i)}}\dd X\\
        &+\frac{1}{N}\int_{\T^N}\rho_{N,t}(X)\log{\frac{\prod_{i=1}^{N}\phi_{\bar{\rho}_t}'(x_i)}{|D\Phi|(X)}}\dd X\\
        =&\I+\II+\III.
    \end{align*}
    Then we are going to estimate the three terms above respectively.\par
    For $\I$, we denote $f_t:=-\log{\bar{\rho}_t\circ\phi_{\bar{\rho}_t}^{-1}}$. Then using Taylor's expansion, we obtain
    \begin{align*}
        \I=&\frac{1}{N}\int_{\T^N}\rho_{N,t}(X)\sum_{i=1}^{N}\left(f_t(\Phi_i(X))-f_t(\phi_{\bar{\rho}_t}(x_i))\right)\dd X\\
        =&\frac{1}{N}\int_{\T^N}\rho_{N,t}(X)\sum_{i=1}^{N}f_t'(\phi_{\bar{\rho}_t}(x_i))\left(\Phi_i(X)-\phi_{\bar{\rho}_t}(x_i)\right)\dd X\\
        &+\frac{1}{N}\int_{\T^N}\rho_{N,t}(X)\sum_{i=1}^{N}\frac{1}{2}f_t''(\xi_i(X))\left(\Phi_i(X)-\phi_{\bar{\rho}_t}(x_i)\right)^2\dd X\\
        =&\I_1+\I_2,
    \end{align*}
    where $\xi_i(X)$ is some intermediate point between $\phi_{\bar{\rho}_t}(x_i)$ and $\Phi_i(X)$. Note that
    $$\Phi_i(X)-\phi_{\bar{\rho}_t}(x_i)=\frac{\Delta}{N}\sum_{j=1}^{N}h(x_i,x_j)-\Delta\int_{\T}h(x_i,y)\dd\bar{\rho}_t(y)=\frac{\Delta}{N}\sum_{j=1}^{N}\bar{h}_t(x_i,x_j),$$
    where
    $$\bar{h}_t(x,y)=h(x,y)-\int_{\T}h(x,z)\bar{\rho}_t(z)\dd z.$$
    Obviously $\bar{h}_t$ satisfies the cancellation rule
    $$\int_{\T}\bar{h}_t(x,y)\bar{\rho}_t(y)\dd y=0,\quad\forall x\in\T.$$
    We further simplify that
    $$\I_1=\frac{\Delta}{N^2}\int_{\T^N}\rho_{N,t}(X)\sum_{i,j=1}^{N}f_t'(\phi_{\bar{\rho}_t}(x_i))\bar{h}_t(x_i,x_j)\dd X=\frac{1}{N^2}\int_{\T^N}\rho_{N,t}(X)\sum_{i,j=1}^{N}K_t^1(x_i,x_j)\dd X,$$
    where
    $$K_t^1(x,y)=\Delta f_t'(\phi_{\bar{\rho}_t}(x))\bar{h}_t(x,y)=-\Delta\frac{\bar{\rho}_t'(x)\bar{h}_t(x,y)}{\bar{\rho}_t(x)\phi_{\bar{\rho}_t}'(x)}.$$
    Under the assumptions on $\bar{\rho}_t$ and $h$, we can easily obtain that $\|f_t''\|_{L^{\infty}}\leq C$. Thus by the Gibbs' inequality(Proposition~\ref{prop:gibbs}), we have
    \begin{align*}
        \I_2&\leq\frac{C\Delta^2}{N^3}\sum_{i=1}^{N}\int_{\T^N}\rho_{N,t}(X)\left(\sum_{j=1}^{N}\bar{h}_t(x_i,x_j)\right)^2\dd X\\
        &=\frac{C}{N^2}\sum_{i=1}^{N}\int_{\T^N}\rho_{N,t}(X)\left(\frac{\Delta^2}{N}\sum_{j,k=1}^{N}\bar{h}_t(x_i,x_j)\bar{h}_t(x_i,x_k)\right)\dd X\\
        &\leq\frac{C}{N^2}\sum_{i=1}^{N}\left[H(\rho_{N,t}|\bar{\rho}_t^{\otimes N})+\log{\int_{\T^N}\bar{\rho}_t^{\otimes N}(X)\exp{\left(\frac{\Delta^2}{N}\sum_{j,k=1}^{N}\bar{h}_t(x_i,x_j)\bar{h}_t(x_i,x_k)\right)}\dd X}\right]\\
        &=C\HN(t)+\frac{C}{N}\log{\int_{\T^N}}\bar{\rho}_t^{\otimes N}(X)\exp{\left(\frac{\Delta^2}{N}\sum_{j,k=1}^{N}\bar{h}_t(x_1,x_j)\bar{h}_t(x_1,x_k)\right)}\dd X.
    \end{align*}
    The last equality is owing to the fact that $\bar{\rho}_t^{\otimes N}$ is  a symmetric distribution. Note that $\bar{h}_t$ satisfies the cancellation rule \eqref{eq:cancellation2}. Hence applying Proposition~\ref{prop:exp2}, we obtain that for $\displaystyle\Delta<\frac{C_2}{2\|h\|_{L^{\infty}}}$, one has
    $$\I_2\leq C\left(\HN(t)+\frac{1}{N}\right),$$
    where $C_2$ is the constant in Proposition~\ref{prop:exp2}.\par
    For $\II$, it has the same structure as $\I$, with $f_t$ replaced by $g_t:=\log{\phi_{\bar{\rho}_t}'\circ\phi_{\bar{\rho}_t}^{-1}}$. So we can similarly obtain that
    \begin{align*}
        \II=&\frac{1}{N}\int_{\T^N}\rho_{N,t}(X)\sum_{i=1}^{N}g_t'(\phi_{\bar{\rho}_t}(x_i))\left(\Phi_i(X)-\phi_{\bar{\rho}_t}(x_i)\right)\dd X\\
        &+\frac{1}{N}\int_{\T^N}\rho_{N,t}(X)\sum_{i=1}^{N}\frac{1}{2}g_t''(\eta_i(X))\left(\Phi_i(X)-\phi_{\bar{\rho}_t}(x_i)\right)^2\dd X\\
        =&\II_1+\II_2,
    \end{align*}
    where $\eta_i(X)$ is some intermediate point between $\phi_{\bar{\rho}_t}(x_i)$ and $\Phi_i(X)$. Moreover,
    $$\II_1=\frac{1}{N^2}\int_{\T^N}\rho_{N,t}(X)\sum_{i,j=1}^{N}K_t^2(x_i,x_j)\dd X,$$
    where
    $$K_t^2(x,y)=\Delta g_t'(\phi_{\bar{\rho}_t}(x))\bar{h}_t(x,y)=\Delta\frac{\phi_{\bar{\rho}_t}''(x)\bar{h}_t(x,y)}{(\phi_{\bar{\rho}_t}'(x))^2},$$
    and
    $$\II_2\leq C\left(\HN(t)+\frac{1}{N}\right),$$
    when $\displaystyle\Delta<\frac{C_2}{2\|h\|_{L^{\infty}}}$.\par
    For $\III$, we first write down the entries of $D\Phi$:
    $$(D\Phi(X))_{ij}=\begin{cases}
    1+\dfrac{\Delta}{N}\displaystyle\sum_{k=1}^N \partial_1 h(x_i,x_k)+\dfrac{\Delta}{N}\partial_2 h(x_i,x_i),& i=j,\\
    \dfrac{\Delta}{N}\partial_2 h(x_i,x_j),& i\neq j.
    \end{cases}$$
    Then we introduce the diagonal matrix
    $$D_t(X):=\diag\left\{\phi_{\bar{\rho}_t}'(x_1),\dots,\phi_{\bar{\rho}_t}'(x_N)\right\},$$
    and the error matrix
    $$M_t(X):=D\Phi(X)-D_t(X).$$
    Note that
    $$\phi_{\bar{\rho}_t}'(x)=1+\Delta\int_{\T}\partial_1 h(x,y)\bar{\rho}_t(y)\dd y.$$
    Thus the entries of $M_t(X)$ are given by
    $$(M_t(X))_{ij}=\begin{cases}
        \dfrac{\Delta}{N}\displaystyle\sum_{k=1}^N \partial_1 \bar{h}_t(x_i,x_k)+\dfrac{\Delta}{N}\partial_2 h(x_i,x_i),& i=j,\\
        \dfrac{\Delta}{N}\partial_2 h(x_i,x_j),& i\neq j.
    \end{cases}$$
    We further define
    $$B_t(X):=(D_t(X))^{-1}M_t(X),$$
    then we have
    $$|D\Phi|(X)=\det{(M_t(X)+D_t(X))}=\left(\prod_{i=1}^{N}\phi_{\bar{\rho}_t}'(x_i)\right)\det{(I+B_t(X))}.$$
    Thus
    $$\III=-\frac{1}{N}\int_{\T^N}\rho_{N,t}(X)\log{\det{(I+B_t(X))}}\dd X.$$
    By the matrix Taylor's expansion,
    $$\log{\det{(I+B_t(X))}}=\tr{B_t(X)}+\mathcal{R}^{(2)}_t(X),$$
    where $|\mathcal{R}^{(2)}_t(X)|\leq C\|B_t(X)\|_{\HS}^2$. Here for a matrix $A=(a_{ij})_{N\times N}$, we let 
    $$\|A\|_{\HS}:=\left(\sum_{i,j=1}^{N}|a_{ij}|^2\right)^{\frac{1}{2}}$$
    denote the Hilbert-Schmidt norm of $A$. With this expansion, we can continue with
    \begin{align*}
        \III\leq&-\frac{\Delta}{N^2}\int_{\T^N}\rho_{N,t}(X)\sum_{i,j=1}^{N}\frac{\partial_1 \bar{h}_t(x_i,x_j)}{\phi_{\bar{\rho}_t}'(x_i)}\dd X-\frac{\Delta}{N^2}\int_{\T^N}\rho_{N,t}(X)\sum_{i=1}^{N}\frac{\partial_2 h(x_i,x_i)}{\phi_{\bar{\rho}_t}'(x_i)}\dd X\\
        &+\frac{C\Delta^2}{N^3}\sum_{i=1}^{N}\int_{\T^N}\rho_{N,t}(X)\left(\sum_{j=1}^N \frac{\partial_1 \bar{h}_t(x_i,x_j)}{\phi_{\bar{\rho}_t}'(x_i)}\right)^2\dd X+\frac{C\Delta^2}{N^3}\int_{\T^N}\rho_{N,t}(X)\sum_{i,j=1}^{N}\left(\frac{\partial_2 h(x_i,x_j)}{\phi_{\bar{\rho}_t}'(x_i)}\right)^2\dd X\\
        &+\frac{2C\Delta^2}{N^3}\int_{\T^N}\rho_{N,t}(X)\sum_{i,j=1}^{N}\frac{\partial_1 \bar{h}_t(x_i,x_j)\cdot\partial_2 h(x_i,x_i)}{\left(\phi_{\bar{\rho}_t}'(x_i)\right)^2}\dd X\\
        \leq&-\frac{\Delta}{N^2}\int_{\T^N}\rho_{N,t}(X)\sum_{i,j=1}^{N}\frac{\partial_1 \bar{h}_t(x_i,x_j)}{\phi_{\bar{\rho}_t}'(x_i)}\dd X+\frac{C\Delta^2}{N^3}\sum_{i=1}^{N}\int_{\T^N}\rho_{N,t}(X)\left(\sum_{j=1}^N \frac{\partial_1 \bar{h}_t(x_i,x_j)}{\phi_{\bar{\rho}_t}'(x_i)}\right)^2\dd X+\frac{C}{N}\\
        =&\III_1+\III_2+\frac{C}{N}.
    \end{align*}
    For $\III_1$, we can rewrite it as
    $$\III_1=\frac{1}{N^2}\int_{\T^N}\rho_{N,t}(X)\sum_{i,j=1}^{N}K_t^3(x_i,x_j)\dd X,$$
    where
    $$K_t^3(x,y)=-\Delta\frac{\partial_1 \bar{h}_t(x,y)}{\phi_{\bar{\rho}_t}'(x)}.$$
    For $\III_2$, it has the same structure as $\I_2$ and $\II_2$. Hence, applying Proposition~\ref{prop:exp2} again, we have the estimate
    $$\III_2\leq C\left(\HN(t)+\frac{1}{N}\right),$$
    when $\displaystyle\Delta<\frac{\min\{1,C_2\}}{4\|\partial_1 h\|_{L^{\infty}}}$.\par
    Combining all the estimates above, we now obtain
    $$R_t\leq \I_1+\II_1+\III_1+C\left(\HN(t)+\frac{1}{N}\right).$$
    Note that
    $$\I_1+\II_1+\III_1=\frac{1}{N^2}\int_{\T^N}\rho_{N,t}(X)\sum_{i,j=1}^{N}K_t(x_i,x_j)\dd X,$$
    where
    $$K_t(x,y)=\sum_{i=1}^{3}K_t^i(x,y)=-\Delta\left[\frac{\bar{\rho}_t'(x)\bar{h}_t(x,y)}{\bar{\rho}_t(x)\phi_{\bar{\rho}_t}'(x)}-\frac{\phi_{\bar{\rho}_t}''(x)\bar{h}_t(x,y)}{(\phi_{\bar{\rho}_t}'(x))^2}+\frac{\partial_1 \bar{h}_t(x,y)}{\phi_{\bar{\rho}_t}'(x)}\right].$$
    One can check that $K_t$ satisfies the two cancellation rules \eqref{eq:cancellation1}. In fact,
    $$\int_{\T}K_t(x,y)\bar{\rho}_t(y)\dd y=0,\quad\forall x\in\T,$$
    since $\bar{h}_t$ satisfies the cancellation rule. And
    \begin{align*}
        \int_{\T}K_t(x,y)\bar{\rho}_t(x)\dd x&=-\Delta\int_{\T}\left[\frac{\bar{\rho}_t'(x)\bar{h}_t(x,y)}{\phi_{\bar{\rho}_t}'(x)}-\frac{\bar{\rho}_t(x)\phi_{\bar{\rho}_t}''(x)\bar{h}_t(x,y)}{(\phi_{\bar{\rho}_t}'(x))^2}+\frac{\bar{\rho}_t(x)\partial_1 \bar{h}_t(x,y)}{\phi_{\bar{\rho}_t}'(x)}\right]\dd x\\
        &=-\Delta\int_{\T}\frac{\partial}{\partial x}\left(\frac{\bar{\rho}_t(x)\bar{h}_t(x,y)}{\phi_{\bar{\rho}_t}'(x)}\right)\dd x=0,\quad\forall y\in\T.
    \end{align*}
    Therefore, by the Gibbs' inequality(Proposition~\ref{prop:gibbs}) and Proposition~\ref{prop:exp1}, we have
    \begin{align*}
        \I_1+\II_1+\III_1&\leq\frac{1}{N}\left[H(\rho_{N,t}|\bar{\rho}_t^{\otimes N})+\log{\int_{\T^N}\bar{\rho}_t^{\otimes N}(X)\exp{\left(\frac{1}{N}\sum_{i,j=1}^{N}K_t(x_i,x_j)\right)}\dd X}\right]\\
        &=\HN(t)+\frac{1}{N}\log{\int_{\T^N}\bar{\rho}_t^{\otimes N}(X)\exp{\left(\frac{1}{N}\sum_{i,j=1}^{N}K_t(x_i,x_j)\right)}\dd X}\\
        &\leq\HN(t)+\frac{C}{N},
    \end{align*}
    for $\Delta<\Delta_0(\|h\|_{C^2},\kappa,K)$. With all these estimates, we can now conclude that
    $$R_t\leq C\left(\HN(t)+\frac{1}{N}\right).$$
\end{proof}
\begin{remark}
    We remark here that we can estimate the second order terms $\I_2,\II_2,\III_2$ respectively by Proposition~\ref{prop:exp2}, but we cannot do this for first order terms $\I_1,\II_1,\III_1$. This is because $K_t^i,i=1,2,3$ do not satisfy the cancellation rule with respect to the first variable respectively. It holds only when we are concerned with the sum $K_t$.
\end{remark}
Now the final estimate \eqref{eq:entropy-estimate} follows immediately by induction.
\begin{proof}[Proof of Theorem~\ref{thm:main}]
Combining Proposition~\ref{prop:one-step} and Proposition~\ref{prop:estimate-Rt}, we get
$$\HN(t+1)\leq C\left(\HN(t)+\frac{1}{N}\right),\quad t=0,1,\dots,T-1,$$
for $\Delta<\Delta_0$. Thus
$$\HN(t+1)+\frac{1}{N}\leq C\left(\HN(t)+\frac{1}{N}\right),\quad t=0,1,\dots,T-1.$$
Iterating the inequality above for $t$ times, we get to the conclusion
$$\HN(t)<\HN(t)+\frac{1}{N}\leq\mathrm{e}^{Ct}\left(\HN(0)+\frac{1}{N}\right),\quad t=0,1,\dots,T,$$
with the constant $C$ modified.
\end{proof}
\subsection{Propagation of chaos}\label{sub:propagation-of-chaos}
With the relative entropy estimate \eqref{eq:entropy-estimate}, the Kac's chaos (Corollary~\ref{cor:Kacs-chaos}) follows by a standard argument. For the sake of completeness, we give a short proof here.
\begin{proof}[Proof of Corollary~\ref{cor:Kacs-chaos}]
    Under the assumption \textbf{(A3)}, we have $\rho_{N,0}=\bar{\rho}_0^{\otimes N}$, and thus $\HN(0)=\frac{1}{N}H(\rho_{N,0}|\bar{\rho}_0^{\otimes N})=0$. By Theorem~\ref{thm:main}, we have $\HN(t)\leq\frac{\mathrm{e}^{Ct}}{N}$ and thus $H(\rho_{N,t}|\bar{\rho}_t^{\otimes N})=N\HN(t)$ is uniformly bounded. Let $\rho_{N,t;k}$ be the $k$-marginal of $\rho_{N,t}$. Then by the super-additivity of relative entropy\cite{HM14}, we have
    $$H\left(\rho_{N,t;k}|\bar{\rho}_t^{\otimes k}\right)\leq\frac{k}{N}H\left(\rho_{N,t}|\bar{\rho}_t^{\otimes N}\right)\to 0,\quad (N\to\infty).$$
    Then by the classical Csisz\'{a}r-Kullback-Pinsker inequality\cite[(22.25)]{Vil08}, for fixed $k\in\mathbb{N}$,
    $$W_1\left(\rho_{N,t;k},\bar{\rho}_t^{\otimes k}\right)\leq C\|\rho_{N,t;k}-\bar{\rho}_t^{\otimes k}\|_{TV}\leq C\sqrt{2H\left(\rho_{N,t;k}|\bar{\rho}_t^{\otimes k}\right)}\to 0,\quad (N\to\infty),$$
    where $\|\cdot\|_{TV}$ denotes the total variation and the first inequality is ensured by the fact that $\T^k$ is compact. Then the Kac's chaos follows by an application of Proposition~\ref{prop:topology}.
\end{proof}
\section{Gaussian Fluctuations}\label{sec:Gaussian-fluctuations}
We are concerned with the central limit theory of our system in this section, i.e. the asymptotic behavior of the fluctuation processes $\eta_t^N=\sqrt{N}(\mu_N(t)-\bar{\rho}_t)$. We first formally derive the equation \eqref{eq:etat} satisfied by the limiting process $\eta_t$. Then we make a short discussion on the time evolution operator $Q_{0,t}$, with which we can show the Gaussianity of $\eta_t$. Finally we use the relative entropy method to prove Theorem~\ref{thm:CLT}, i.e. the convergence of $\eta^N$ to $\eta$ in distribution. More precisely, we will use a lemma from \cite{WZZ23} to bound the negative Sobolev norms of $\eta_t^N$ by the relative entropy, which implies the tightness of laws of $\eta^N$. Then actually the only thing we need to show is that any tight limit of $\eta^N$ is indeed a solution to \eqref{eq:etat}.
\subsection{Formal derivation of the equation for $\eta_t$}\label{sub:formal-derivation}
We first derive the equation \eqref{eq:etat} on a heuristic level. In Section~\ref{sub:convergence} we will rigorously justify the heuristic assumptions we make in the process of formal derivation.\par
According to \cite[Theorem 1]{FG15}, if the initial data $\{x_i(0)\}_{i=1}^{\infty}$ are i.i.d. with common distribution $\bar{\rho}_0$, then $\mathbb{E}W_1(\mu_N(0),\bar{\rho}_0)=O(N^{-1/2})$. Thus we formally assume that $\displaystyle W_1(\mu_N(0),\bar{\rho}_0)=O(N^{-1/2})$ almost surely. Then by Theorem \ref{thm:Dobrushin}, we have $\displaystyle W_1(\mu_N(t),\bar{\rho}_t)=O(N^{-1/2})$ for any $t\geq 0$. We consider
$$\eta_{t+1}^N=\sqrt{N}(\mu_N(t+1)-\bar{\rho}_{t+1})=\sqrt{N}\left(\left(F_{\mu_N(t)}\right)_{\ast}\mu_N(t)-\left(F_{\bar{\rho}_t}\right)_{\ast}\bar{\rho}_t\right).$$
Then for any test function $\varphi\in C^{\infty}(\T)$, we have
\begin{align*}
    \langle\varphi,\eta_{t+1}^N\rangle=&\sqrt{N}\langle\varphi,\left(F_{\mu_N(t)}\right)_{\ast}\mu_N(t)-\left(F_{\bar{\rho}_t}\right)_{\ast}\bar{\rho}_t\rangle\\
    =&\sqrt{N}\langle\varphi,\left(F_{\mu_N(t)}\right)_{\ast}\mu_N(t)-\left(F_{\mu_N(t)}\right)_{\ast}\bar{\rho}_t\rangle+\sqrt{N}\langle\varphi,\left(F_{\mu_N(t)}\right)_{\ast}\bar{\rho}_t-\left(F_{\bar{\rho}_t}\right)_{\ast}\bar{\rho}_t\rangle\\
    =&\langle\varphi,\left(F_{\mu_N(t)}\right)_{\ast}\eta_t^N\rangle+\sqrt{N}\langle\varphi\circ F\circ\phi_{\mu_N(t)}-\varphi\circ F\circ\phi_{\bar{\rho}_t},\bar{\rho}_t\rangle.
\end{align*}
Note that
$$\phi_{\mu_N(t)}(x)-\phi_{\bar{\rho}_t}(x)=\Delta\int_{\T}h(x,y)\dd(\mu_N(t)-\bar{\rho}_t)(y)=\Delta(h\ast(\mu_N(t)-\bar{\rho}_t))(x),$$
where for a measure $\mu$ on $\T$, we denote
$$(h\ast\mu)(x):=\int_{\T}h(x,y)\dd\mu(y).$$
Thus
$$|\phi_{\mu_N(t)}-\phi_{\bar{\rho}_t}|\leq\Delta\Vert\partial_2h\Vert_{L^{\infty}}W_1(\mu_N(t),\bar{\rho}_t)=O\left(\frac{1}{\sqrt{N}}\right).$$
And by Taylor's expansion, we can continue with
\begin{equation}\label{eq:weak-etatN}
\begin{aligned}
    \langle\varphi,\eta_{t+1}^N\rangle=&\langle\varphi,\left(F_{\mu_N(t)}\right)_{\ast}\eta_t^N\rangle+\Delta\langle(\varphi\circ F)'(\phi_{\bar{\rho}_t})\cdot h\ast\eta_t^N,\bar{\rho}_t\rangle\\
    &\hspace{0.5cm}+\frac{\Delta^2\sqrt{N}}{2}\langle(\varphi\circ F)''(\xi)\cdot(\phi_{\mu_N(t)}-\phi_{\bar{\rho}_t})^2,\bar{\rho}_t\rangle\\
    =&\langle\varphi,\left(F_{\mu_N(t)}\right)_{\ast}\eta_t^N\rangle+\Delta\langle(\varphi\circ F)'(\phi_{\bar{\rho}_t})\cdot h\ast\eta_t^N,\bar{\rho}_t\rangle+O\left(\frac{1}{\sqrt{N}}\right).
\end{aligned}
\end{equation}    
Letting $N\to\infty$, we formally obtain
\begin{equation}\label{eq:weak-etat}
    \langle\varphi,\eta_{t+1}\rangle=\langle\varphi,\left(F_{\bar{\rho}_t}\right)_{\ast}\eta_t\rangle+\Delta\langle(\varphi\circ F)'(\phi_{\bar{\rho}_t})\cdot h\ast\eta_t,\bar{\rho}_t\rangle,
\end{equation}
which is the weak form of the equation \eqref{eq:etat}.
\subsection{Time evolution operator $Q_{0,t}$ and Gaussianity of $\eta$}\label{sub:time-evolution}
Recall that the assumption \textbf{(A3)} is that the initial data $\{x_i(0)\}_{i=1}^{\infty}$ are i.i.d. with common distribution $\bar{\rho}_0$. Under this assumption, it is well-known that there exists a centered Gaussian field $\eta_0$, which is a random element on the space of tempered distributions $\mathcal{S}'(\T)$, such that the sequence $\{\eta_0^N\}_{N=1}^{\infty}$ converges in distribution to $\eta_0$ in $\mathcal{S}'(\T)$.\par
We now briefly discuss how to characterize this $\eta_0$. Indeed, the characterization \eqref{eq:Gaussian} using a single test function cannot completely characterize the properties of $\eta_0$. What we need is a multi-dimensional version of \eqref{eq:Gaussian}. In fact, according to the multi-dimensional version of the Central Limit Theorem, for any $m$ test functions $\varphi_1,\dots,\varphi_m\in C^{\infty}(\T)$, we have
$$\begin{aligned}
    \left(\left\langle\varphi_1,\eta_0^N\right\rangle,\dots,\left\langle\varphi_m,\eta_0^N\right\rangle\right)&=\frac{1}{\sqrt{N}}\sum_{i=1}^{N}\left(\varphi_1(x_i(0))-\langle\varphi_1,\bar{\rho}_0\rangle,\dots,\varphi_m(x_i(0))-\langle\varphi_m,\bar{\rho}_0\rangle\right)\\
    &\to\mathcal{N}\left(0,\Sigma\right)
\end{aligned}$$
in distribution as $N\to\infty$, where $\Sigma$ is an $m\times m$ matrix with coefficients
$$\Sigma_{ij}=\langle\varphi_i\varphi_j,\bar{\rho}_0\rangle-\langle\varphi_i,\bar{\rho}_0\rangle\langle\varphi_j,\bar{\rho}_0\rangle,\quad 1\leq i,j\leq m,$$
and $\mathcal{N}(0,\Sigma)$ denotes the $m$-dimensional centered Gaussian distribution with covariance matrix $\Sigma$. Thus a complete characterization of $\eta_0$ is given by
\begin{equation}\label{eq:multi-Gaussian}
    \left(\left\langle\varphi_1,\eta_0\right\rangle,\dots,\left\langle\varphi_m,\eta_0\right\rangle\right)\sim\mathcal{N}\left(0,\Sigma\right),\quad m\geq 1,\ \varphi_1,\dots,\varphi_m\in C^{\infty}(\T),
\end{equation}
where the covariance matrix $\Sigma$ is as described above.\par
Starting from $\eta_0$ and evolving according to \eqref{eq:etat}, we will get a solution $\eta=(\eta_t)_{0\leq t\leq T}$ of \eqref{eq:etat}. One can check that $\eta$ is still a centered Gaussian field. To see this, we need to introduce the time evolution operator $Q_{0,t}$, which is indeed the adjoint of the operator $\eta_0\mapsto\eta_t$. Actually, given any $\varphi\in C^{\infty}(\T)$, we deduce from \eqref{eq:etat} that
\begin{align*}
    \langle\varphi,\eta_{t+1}\rangle&=\langle\varphi,\left(F_{\bar{\rho}_t}\right)_{\ast}\eta_t\rangle+\Delta\langle(\varphi\circ F)'(\phi_{\bar{\rho}_t})\cdot h\ast\eta_t,\bar{\rho}_t\rangle\\
    &=\langle\varphi\circ F_{\bar{\rho}_t},\eta_t\rangle+\Delta\int_{\T}(\varphi\circ F)'(\phi_{\bar{\rho}_t}(x))\bar{\rho}_t(x)\left(\int_{\T}h(x,y)\dd\eta_t(y)\right)\dd x\\
    &=\int_{\T}\left((\varphi\circ F_{\bar{\rho}_t})(y)+\Delta\int_{\T}(\varphi\circ F)'(\phi_{\bar{\rho}_t}(x))\bar{\rho}_t(x)h(x,y)\dd x\right)\dd\eta_t(y).
\end{align*}
Thus if we define
$$(Q_{t,t+1}\varphi)(y):=(\varphi\circ F_{\bar{\rho}_t})(y)+\Delta\int_{\T}(\varphi\circ F)'(\phi_{\bar{\rho}_t}(x))\bar{\rho}_t(x)h(x,y)\dd x,\quad y\in\T,\ t\geq 0,$$
and
\begin{equation}\label{eq:def-Qst}
    Q_{s,t}:=Q_{s,s+1}\circ\cdots\circ Q_{t-1,t},\quad 0\leq s\leq t,
\end{equation}
then
$$\langle\varphi,\eta_{t+1}\rangle=\langle Q_{t,t+1}\varphi,\eta_t\rangle,$$
and hence
$$\langle\varphi,\eta_t\rangle=\langle Q_{0,t}\varphi,\eta_0\rangle\sim\mathcal{N}\left(0,\langle|Q_{0,t}\varphi|^2,\bar{\rho}_0\rangle-\langle Q_{0,t}\varphi,\bar{\rho}_0\rangle^2\right).$$
To see the Gaussianity of $\eta$, we take $T+1$ test functions $\phi_0,\dots,\phi_T\in C^{\infty}(\T)$. Then by the discussion above, we have
$$\left(\langle\varphi_0,\eta_0\rangle,\dots,\langle\varphi_T,\eta_T\rangle\right)=\left(\langle\varphi_0,\eta_0\rangle,\langle Q_{0,1}\varphi_1,\eta_0\rangle,\dots,\langle Q_{0,T}\varphi_T,\eta_0\rangle\right)\sim\mathcal{N}\left(0,\Gamma\right),$$
where $\Gamma$ is a $(T+1)\times(T+1)$ matrix with coefficients
$$\Gamma_{st}=\langle(Q_{0,s}\varphi_s)(Q_{0,t}\varphi_t),\bar{\rho}_0\rangle-\langle Q_{0,s}\varphi_s,\bar{\rho}_0\rangle\langle Q_{0,t}\varphi_t,\bar{\rho}_0\rangle,\quad 0\leq s,t\leq T.$$
\subsection{Convergence of the fluctuation processes}\label{sub:convergence}
Now we are going to prove the main part of Theorem~\ref{thm:CLT}, which is the convergence of $\eta^N$ to $\eta$ in distribution in $(H^{-\alpha}(\T))^{T+1},\alpha>1/2$.\par
Our crucial lemma is from \cite{WZZ23}, which bounds the $H^{-\alpha}$ norm of $\eta_t^N$ uniformly by the relative entropy.
\begin{lemma}[Wang, Zhao and Zhu {\cite[Lemma 2.6]{WZZ23}}]\label{lem:bound}
    For each $\alpha>1/2$, there exists a constant $C_{\alpha}>0$ such that for $0\leq t\leq T$,
    \begin{equation}\label{eq:bound-negative-Sobolev}
        \mathbb{E}\Vert\mu_N(t)-\bar{\rho}_t\Vert_{H^{-\alpha}}^2\leq\frac{C_{\alpha}}{N}(H(\rho_{N,t}|\bar{\rho}_t^{\otimes N})+1).
    \end{equation}
\end{lemma}
In particular, Lemma \ref{lem:bound} implies the tightness of laws of $\{\eta_0^N\}$ in $H^{-\alpha}$, which together with the Assumption \textbf{(A3)} yields the convergence of $\{\eta_0^N\}$ in negative Sobolev spaces.
\begin{corollary}\label{cor:tightconvergence}
    For every $\alpha>1/2$, $\eta_0^N$ converges in distribution to $\eta_0$ in $H^{-\alpha}(\T)$.
\end{corollary}
\begin{proof}
    We first prove the tightness of laws of $\{\eta_0^N\}$ in $H^{-\alpha}$. Indeed, fix a $\beta\in(1/2,\alpha)$, the Sobolev Embedding Theorem says that $H^{-\beta}(\T)$ is compactly embedded into $H^{-\alpha}(\T)$. Hence for any $R>0$, the ball
    $$B_R^\beta:=\{\nu|\ \Vert\nu\Vert_{H^{-\beta}}\leq R\}$$
    is a compact set in $H^{-\alpha}(\T)$. Besides, due to the Assumption \textbf{(A3)}, we have $\rho_{N,0}=\bar{\rho}_0^{\otimes N}$ and thus $H(\rho_{N,0}|\bar{\rho}_0^{\otimes N})=0$. Therefore, we obtain by Markov's inequality and Lemma \ref{lem:bound} that
    $$\mathcal{L}_{\eta_0^N}((B_R^\beta)^c)=\mathbb{P}(\Vert\eta_0^N\Vert_{H^{-\beta}}>R)\leq\frac{1}{R^2}\mathbb{E}\Vert\eta_0^N\Vert_{H^{-\beta}}^2=\frac{N}{R^2}\mathbb{E}\Vert\mu_N(0)-\bar{\rho}_0\Vert_{H^{-\beta}}^2\leq\frac{C_{\beta}}{R^2},$$
    which implies the tightness of laws of $\{\eta_0^N\}$ in $H^{-\alpha}$. Here for a random element $X$, we use $\mathcal{L}_X$ to denote the distribution of $X$.\par
    By the Prohorov's Theorem, for every subsequence of $\{\eta_0^N\}$, there exists a further subsequence converging in distribution to some random element in $H^{-\alpha}(\T)$. But we have known that $\{\eta_0^N\}$ converges in distribution to $\eta_0$ in $\mathcal{S}'(\T)$, so we must have $\eta_0^N$ converges in distribution to $\eta_0$ in $H^{-\alpha}(\T)$.
\end{proof}
By the same arguments as in the proof of Corollary \ref{cor:tightconvergence}, we can show the tightness of laws of $\{\eta^N\}$ in $(H^{-\alpha}(\T))^{T+1}$. Since the uniqueness of the solution to \eqref{eq:etat} is obvious, we only need to show that every tight limit of $\{\eta^N\}$ in $(H^{-\alpha}(\T))^{T+1}$ is a solution to \eqref{eq:etat}. And by the argument below, we only need to prove this statement for a fixed $\alpha_0>1/2$. Our choice is $\alpha_0=3$.
\begin{theorem}\label{thm:solution}
    If a subsequence $\{\eta^{N_k}\}$ of $\{\eta^N\}$ converges in distribution to $\tilde{\eta}$ in $(H^{-3}(\T))^{T+1}$, then $\tilde{\eta}$ is a solution to \eqref{eq:etat}.
\end{theorem}
Before giving the proof of Theorem~\ref{thm:solution}, we first use it to complete the proof of Theorem~\ref{thm:CLT}.
\begin{proof}[Proof of Theorem \ref{thm:CLT}]
    We first assume that $\alpha\in(1/2,3]$. By Theorem \ref{thm:main} and Lemma \ref{lem:bound}, the relative entropy $H(\rho_{N,t}|\bar{\rho}_t^{\otimes N})$, and thus $\mathbb{E}\Vert\eta_t^N\Vert_{H^{-\alpha}}^2$, are uniformly bounded. So by the same arguments as Corollary \ref{cor:tightconvergence}, we can show the tightness of laws of $\{\eta^N\}$ in $(H^{-\alpha}(\T))^{T+1}$. By the Prohorov's Theorem again, for every subsequence of $\{\eta^N\}$, there exists a further subsequence converging in distribution to some random element $\tilde{\eta}$ in $(H^{-\alpha}(\T))^{T+1}$. Since $(H^{-\alpha}(\T))^{T+1}$ is embedded into $(H^{-3}(\T))^{T+1}$, $\{\eta^N\}$ also converges in distribution to $\tilde{\eta}$ in $(H^{-3}(\T))^{T+1}$. Thus $\tilde{\eta}$ is a solution to \eqref{eq:etat} by Theorem \ref{thm:solution}. By Corollary \ref{cor:tightconvergence}, we have $\tilde{\eta}_0\overset{d}{=}\eta_0$. Hence $\tilde{\eta}\overset{d}{=}\eta$ by the uniqueness of solution to \eqref{eq:etat}. Therefore, $\{\eta^N\}$ converges in distribution to $\eta$ in $(H^{-\alpha}(\T))^{T+1}$ for $\alpha\in(1/2,3]$, and hence for $\alpha\in(3,+\infty)$ as well. The Gaussianity of $\eta$ comes from the discussion in Section~\ref{sub:time-evolution}.
\end{proof}
Now the only thing we need to do is to prove Theorem~\ref{thm:solution}. Without loss of generality, we assume that $\{\eta^N\}$ converges in distribution to $\tilde{\eta}$ in $(H^{-3}(\T))^{T+1}$. By the Skorohod representation theorem, there exists $\tilde{\eta}^N$ with the same distribution as $\eta^N$ in $(H^{-3}(\T))^{T+1}$, $N\geq 1$, such that $\{\tilde{\eta}^N\}$ converges almost surely. We still denote the almost sure limit by $\tilde{\eta}$. We will use the following proposition to identify $\tilde{\eta}$ with the solution to \eqref{eq:etat} starting from initial data $\tilde{\eta}_0$.
\begin{proposition}\label{prop:identify}
    Let $\hat{\eta}=(\hat{\eta}_t)_{t=0}^{T}$ be the solution to \eqref{eq:etat} with initial condition $\hat{\eta}_0=\tilde{\eta}_0$. Then there exists a constant $C>0$, such that
    \begin{equation}\label{eq:identify}
        \mathbb{E}\Vert\tilde{\eta}_t^N-\hat{\eta}_t\Vert_{H^{-3}}\leq\mathrm{e}^{Ct}\left(\mathbb{E}\Vert\tilde{\eta}_0^N-\tilde{\eta}_0\Vert_{H^{-3}}+\frac{1}{\sqrt{N}}\right),\quad t=0,\dots,T.
    \end{equation}
\end{proposition}
To prove this proposition, we need the following lemma from \cite{BL19}, which is used to justify the heuristic assumption we made in the formal derivation.
\begin{lemma}[Bobkov and Ledoux, {\cite[Theorem 5.1]{BL19}}]\label{lem:expectation-Wasserstein}
    Define
    \begin{equation}\label{eq:def-J2}
        J_2(\bar{\rho}_0):=\int_0^1\frac{P(x)(1-P(x))}{\bar{\rho}_0(x)}\dd x,
    \end{equation}
    where $P(x)=\int_0^x\bar{\rho}_0(y)\dd y$ is the distribution function associated with $\bar{\rho}_0$. Then we have
    \begin{equation}\label{eq:expectation-Wasserstein}
        \mathbb{E}W_2^2(\mu_N(0),\bar{\rho}_0)\leq\frac{2}{N+1}J_2(\bar{\rho}_0).
    \end{equation}
\end{lemma}
The following lemma shows that $\varphi\mapsto\varphi\circ F_{\bar{\rho}_t}$ is a bounded operator in $H^3(\T)$.
\begin{lemma}\label{lem:H3}
    For any $\varphi\in H^3(\T)$, we have
    \begin{equation}\label{eq:H3}
        \Vert\varphi\circ F_{\bar{\rho}_t}\Vert_{H^3}\leq C\Vert\varphi\Vert_{H^3},
    \end{equation}
    where the constant $C>0$ depends on $F,\Delta$ and $\| h\|_{C^3}$.
\end{lemma}
\begin{proof}
    Since for $\Delta<\Delta_1\leq\Delta_0$, $\phi_{\bar{\rho}_t}$ is a diffeomorphism on $\T$, if we let 
    $$m:=\int_0^1|F_{\bar{\rho}_t}'(x)|\dd x=\int_0^1|F'(x)|\dd x\in\mathbb{N},$$
    then the inverse of $F_{\bar{\rho}_t}$ will have $m$ branches, denoted by $\tau_1,\dots,\tau_m$. Let $A_i=\tau_i(\T),\ 1\leq i\leq m$, then we have
    \begin{align*}
        \Vert\varphi\circ F_{\bar{\rho}_t}\Vert_{L^2}^2&=\int_{\T}|\varphi\circ F_{\bar{\rho}_t}(x)|^2\dd x\\
        &=\sum_{i=1}^{m}\int_{A_i}|\varphi\circ F_{\bar{\rho}_t}(x)|^2\dd x\\
        &=\sum_{i=1}^{m}\int_{\T}\frac{|\varphi(y)|^2}{|F_{\bar{\rho}_t}'(\tau_i(y)|}\dd y\\
        &\leq\frac{m}{c(1-\Delta\Vert\partial_1h\Vert_{L^{\infty}})}\Vert\varphi\Vert_{L^2}^2.
    \end{align*}
    Since the derivatives of $\varphi\circ F_{\bar{\rho}_t}$ of orders up to $3$ are given by
    \begin{align*}
        (\varphi\circ F_{\bar{\rho}_t})'&=(\varphi'\circ F_{\bar{\rho}_t})\cdot F_{\bar{\rho}_t}',\\
        (\varphi\circ F_{\bar{\rho}_t})''&=(\varphi''\circ F_{\bar{\rho}_t})\cdot(F_{\bar{\rho}_t}')^2+(\varphi'\circ F_{\bar{\rho}_t})\cdot F_{\bar{\rho}_t}''\\
        (\varphi\circ F_{\bar{\rho}_t})'''&=(\varphi'''\circ F_{\bar{\rho}_t})\cdot(F_{\bar{\rho}_t}')^3+3(\varphi''\circ F_{\bar{\rho}_t})\cdot F_{\bar{\rho}_t}'\cdot F_{\bar{\rho}_t}''+(\varphi'\circ F_{\bar{\rho}_t})\cdot F_{\bar{\rho}_t}''',
    \end{align*}
    and $F_{\bar{\rho}_t}$ is $C^3$ according to Assumption \textbf{(A1.3)}, a similar argument as above yields that
    $$\Vert(\varphi\circ F_{\bar{\rho}_t})^{(k)}\Vert_{L^2}^2\leq C\Vert\varphi^{(k)}\Vert_{L^2}^2,\quad k=0,1,2,3.$$
    Therefore, the estimate \eqref{eq:H3} follows from Proposition~\ref{prop:Hn}:
    $$\Vert\varphi\circ F_{\bar{\rho}_t}\Vert_{H^3}^2\leq C\sum_{k=0}^{3}\Vert(\varphi\circ F_{\bar{\rho}_t})^{(k)}\Vert_{L^2}^2\leq C\sum_{k=0}^{3}\Vert\varphi^{(k)}\Vert_{L^2}^2\leq C\Vert\varphi\Vert_{H^3}^2.$$
\end{proof}
Now we are ready to prove Proposition~\ref{prop:identify}, which is the most delicate part in this section.
\begin{proof}[Proof of Proposition~\ref{prop:identify}]
    We first define $\tilde{\mu}_N(t):=\bar{\rho}_t+\tilde{\eta}_t^N/\sqrt{N}$, then since $\bar{\rho}_t$ is not random, we have $\tilde{\mu}_N\overset{d}{=}\mu_N$ in $(H^{-3}(\T))^{T+1}$. Hence $\tilde{\mu}_N(t+1)=(F_{\tilde{\mu}_N(t)})_{\ast}\tilde{\mu}_N(t)$ and consequently 
    $$\tilde{\eta}_{t+1}^N=\sqrt{N}(\tilde{\mu}_{N}(t+1)-\bar{\rho}_{t+1})=\sqrt{N}\left((F_{\tilde{\mu}_N(t)})_{\ast}\tilde{\mu}_N(t)-(F_{\bar{\rho}_t})_{\ast}\bar{\rho}_t\right)$$
    almost surely for $t=0,\dots,T-1$. Thus for any test function $\varphi\in H^3(\T)$, the first equality of \eqref{eq:weak-etatN} and \eqref{eq:weak-etat} still holds true, with $\eta^N,\mu_N$ and $\eta$ replaced by $\tilde{\eta}^N,\tilde{\mu}_N$ and $\hat{\eta}$ respectively. More precisely, we have
    $$\begin{aligned}
        \langle\varphi,\tilde{\eta}_{t+1}^N\rangle=&\langle\varphi\circ F\circ\phi_{\tilde{\mu}_N(t)},\tilde{\eta}_t^N\rangle+\Delta\langle(\varphi\circ F)'(\phi_{\bar{\rho}_t})\cdot h\ast\tilde{\eta}_t^N,\bar{\rho}_t\rangle\\
        &\hspace{0.5cm}+\frac{\Delta^2\sqrt{N}}{2}\langle(\varphi\circ F)''(\xi)\cdot[h\ast(\tilde{\mu}_N(t)-\bar{\rho}_t)]^2,\bar{\rho}_t\rangle
    \end{aligned}$$
    and
    $$\langle\varphi,\hat{\eta}_{t+1}\rangle=\langle\varphi\circ F\circ\phi_{\bar{\rho}_t},\hat{\eta}_t\rangle+\Delta\langle(\varphi\circ F)'(\phi_{\bar{\rho}_t})\cdot h\ast\hat{\eta}_t,\bar{\rho}_t\rangle$$
    for $t=0,\dots,T-1$. Taking the difference, we get
    \begin{align*}
        &\langle\varphi,\tilde{\eta}_{t+1}^N-\hat{\eta}_{t+1}\rangle\\
        =&\langle\varphi\circ F\circ\phi_{\tilde{\mu}_N(t)}-\varphi\circ F\circ\phi_{\bar{\rho}_t},\tilde{\eta}_t^N\rangle+\langle\varphi\circ F\circ\phi_{\bar{\rho}_t},\tilde{\eta}_t^N-\hat{\eta}_t\rangle\\
        &\hspace{0.5cm}+\Delta\langle(\varphi\circ F)'(\phi_{\bar{\rho}_t})\cdot h\ast(\tilde{\eta}_t^N-\hat{\eta}_t),\bar{\rho}_t\rangle+\frac{\Delta^2\sqrt{N}}{2}\langle(\varphi\circ F)''(\xi)\cdot[h\ast(\tilde{\mu}_N(t)-\bar{\rho}_t)]^2,\bar{\rho}_t\rangle\\
        =&\I+\II+\III+\IV.
    \end{align*}
    By the Sobolev embedding theorem $\varphi\in C^2(\T)$ and $\Vert\varphi\Vert_{C^2}\leq C\Vert\varphi\Vert_{H^3}$. And by Lemma~\ref{lem:H3}, $\Vert\varphi\circ F\circ\phi_{\bar{\rho}_t}\Vert_{H^3}\leq C\Vert\varphi\Vert_{H^3}$. Then we estimate the four terms above respectively:
    \begin{align*}
        |\I|&\leq\sqrt{N}W_1(\tilde{\mu}_N(t),\bar{\rho}_t)\cdot\Vert(\varphi\circ F\circ\phi_{\tilde{\mu}_N(t)}-\varphi\circ F\circ\phi_{\bar{\rho}_t})'\Vert_{L^{\infty}}\\
        &=\sqrt{N}W_1(\tilde{\mu}_N(t),\bar{\rho}_t)\cdot\Vert(\varphi\circ F)'(\phi_{\tilde{\mu}_N(t)})\cdot\phi_{\tilde{\mu}_N(t)}'-(\varphi\circ F)'(\phi_{\bar{\rho}_t})\cdot\phi_{\bar{\rho}_t}'\Vert_{L^{\infty}}\\
        &\leq\sqrt{N}W_1(\tilde{\mu}_N(t),\bar{\rho}_t)\big(\Vert\phi_{\tilde{\mu}_N(t)}'\Vert_{L^{\infty}}\cdot\Vert(\varphi\circ F)'(\phi_{\tilde{\mu}_N(t)})-(\varphi\circ F)'(\phi_{\bar{\rho}_t})\Vert_{L^{\infty}}\\
        &\hspace{4cm}+\Vert(\varphi\circ F)'\Vert_{L^{\infty}}\cdot\Vert\phi_{\tilde{\mu}_N(t)}'-\phi_{\bar{\rho}_t}'\Vert_{L^{\infty}}\big)\\
        &\leq C\sqrt{N}W_1(\tilde{\mu}_N(t),\bar{\rho}_t)\big(\Vert(\varphi\circ F)''\Vert_{L^{\infty}}\cdot\Vert\phi_{\tilde{\mu}_N(t)}-\phi_{\bar{\rho}_t}\Vert_{L^{\infty}}+\Vert(\varphi\circ F)'\Vert_{L^{\infty}}\cdot\Vert\phi_{\tilde{\mu}_N(t)}'-\phi_{\bar{\rho}_t}'\Vert_{L^{\infty}}\big)\\
        &\leq C\sqrt{N}W_1(\tilde{\mu}_N(t),\bar{\rho}_t)\Vert\varphi\Vert_{H^3}\big(\Vert h\ast(\tilde{\mu}_N(t)-\bar{\rho}_t)\Vert_{L^{\infty}}+\Vert (\partial_1h)\ast(\tilde{\mu}_N(t)-\bar{\rho}_t)\Vert_{L^{\infty}}\big)\\
        &\leq C\sqrt{N}W_1^2(\tilde{\mu}_N(t),\bar{\rho}_t)\Vert\varphi\Vert_{H^3},\\
        |\II|&\leq\Vert\varphi\circ F\circ\phi_{\bar{\rho}_t}\Vert_{H^3}\cdot\Vert\tilde{\eta}_t^N-\hat{\eta}_t\Vert_{H^{-3}}\\
        &\leq C\Vert\tilde{\eta}_t^N-\hat{\eta}_t\Vert_{H^{-3}}\cdot\Vert\varphi\Vert_{H^3},\\
        |\III|&\leq C\Vert(\varphi\circ F)'\Vert_{L^{\infty}}\cdot\Vert h\ast(\tilde{\eta}_t^N-\hat{\eta}_t)\Vert_{L^{\infty}}\\
        &\leq C\Vert\tilde{\eta}_t^N-\hat{\eta}_t\Vert_{H^{-3}}\cdot\Vert\varphi\Vert_{H^3},\\
        |\IV|&\leq C\sqrt{N}\Vert(\varphi\circ F)''\Vert_{L^{\infty}}\cdot\Vert h\ast(\tilde{\mu}_N(t)-\bar{\rho}_t)\Vert_{L^{\infty}}^2\\
        &\leq C\sqrt{N}W_1^2(\tilde{\mu}_N(t),\bar{\rho}_t)\Vert\varphi\Vert_{H^3}.
    \end{align*}
    Therefore,
    $$|\langle\varphi,\tilde{\eta}_{t+1}^N-\hat{\eta}_{t+1}\rangle|\leq C(\Vert\tilde{\eta}_t^N-\hat{\eta}_t\Vert_{H^{-3}}+\sqrt{N}W_1^2(\tilde{\mu}_N(t),\bar{\rho}_t))\Vert\varphi\Vert_{H^3},\quad\forall\varphi\in H^3(\T),$$
    implying that
    \begin{equation}\label{eq:H-3-one-step}
        \Vert\tilde{\eta}_{t+1}^N-\hat{\eta}_{t+1}\Vert_{H^{-3}}\leq C(\Vert\tilde{\eta}_t^N-\hat{\eta}_t\Vert_{H^{-3}}+\sqrt{N}W_1^2(\tilde{\mu}_N(t),\bar{\rho}_t)),\quad t=0,\dots,T-1.
    \end{equation}
    By the definition of Wasserstein distances and the Cauchy-Schwarz inequality, we have $W_1(\mu,\nu)\leq W_2(\mu,\nu)$ for any two probability measures $\mu$ and $\nu$. And since $\bar{\rho}_0$ is strictly positive, we have $J_2(\bar{\rho}_0)<\infty$. Thus applying Theorem~\ref{thm:Dobrushin} and Lemma~\ref{lem:expectation-Wasserstein}, we estimate that
    $$\mathbb{E}W_1^2(\tilde{\mu}_N(t),\bar{\rho}_t)\leq\mathrm{e}^{Ct}\mathbb{E}W_1^2(\tilde{\mu}_N(0),\bar{\rho}_0)\leq\mathrm{e}^{Ct}\mathbb{E}W_2^2(\tilde{\mu}_N(0),\bar{\rho}_0)=\mathrm{e}^{Ct}\mathbb{E}W_2^2(\mu_N(0),\bar{\rho}_0)\leq\frac{C\mathrm{e}^{Ct}}{N}.$$
    Hence taking expectation in \eqref{eq:H-3-one-step}, we obtain
    $$\mathbb{E}\Vert\tilde{\eta}_{t+1}^N-\hat{\eta}_{t+1}\Vert_{H^{-3}}\leq C\left(\mathbb{E}\Vert\tilde{\eta}_t^N-\hat{\eta}_t\Vert_{H^{-3}}+\frac{\mathrm{e}^{Ct}}{\sqrt{N}}\right),t=0,\dots,T-1.$$
    Then an induction argument yields the desired estimate \eqref{eq:identify}.
\end{proof}
With Proposition~\ref{prop:identify} in hand, we are not too far from the destination. Using the bound \eqref{eq:bound-negative-Sobolev} again with the properties of uniform integrability, we complete our proof as follows.
\begin{proof}[Proof of Theorem~\ref{thm:solution}]
    For simplicity, we denote $Y_N(t):=\Vert\tilde{\eta}_t^N-\tilde{\eta}_t\Vert_{H^{-3}}$. Then Lemma~\ref{lem:bound} and Theorem~\ref{thm:main} show that $\mathbb{E}Y_N^2(t)$ are uniformly bounded, which implies the uniform integrability of $Y_N(t)$. Together with the fact that $\tilde{\eta}_t^N$ converges to $\tilde{\eta}_t$ in $H^{-3}(\T)$ almost surely, we conclude that
    $$\lim_{N\to\infty}\mathbb{E}Y_N(t)=\lim_{N\to\infty}\mathbb{E}\Vert\tilde{\eta}_t^N-\tilde{\eta}_t\Vert_{H^{-3}}=0,\quad t=0,\dots,T.$$
    In particular, we have $\lim_{N\to\infty}\mathbb{E}\Vert\tilde{\eta}_0^N-\tilde{\eta}_0\Vert_{H^{-3}}=0$. Thus by Proposition~\ref{prop:identify}, we have
    $$\lim_{N\to\infty}\mathbb{E}\Vert\tilde{\eta}_t^N-\hat{\eta}_t\Vert_{H^{-3}}=0,\quad t=0,\dots,T,$$
    which forces that $\tilde{\eta}_t=\hat{\eta}_t$ almost surely for $t=0,\dots,T$. Thus $\tilde{\eta}$ must be a solution to \eqref{eq:etat}.
\end{proof}
\subsection{Quantitative estimate}\label{sub:quantitative}
If we can check the proof of Proposition~\ref{prop:identify} carefully, it can be seen that the estimate \eqref{eq:identify} still holds true with $\tilde{\eta}^N$ and $\hat{\eta}$ replaced by $\eta^N$ and $\eta$ respectively. More explicitly, we have
\begin{equation}\label{eq:quantitative2}
    \mathbb{E}\Vert\eta_t^N-\eta_t\Vert_{H^{-3}}\leq\mathrm{e}^{Ct}\left(\mathbb{E}\Vert\eta_0^N-\eta_0\Vert_{H^{-3}}+\frac{1}{\sqrt{N}}\right),\quad t=0,\dots,T.
\end{equation}
It seems that we have obtained a quantitative result. However, since $\eta_0^N$ converges to $\eta_0$ only in distribution, we have no information about the joint distribution of $\eta_0^N$ and $\eta_0$. So it is impossible and actually unreasonable to calculate $\mathbb{E}\Vert\eta_0^N-\eta_0\Vert_{H^{-3}}$. And in fact, one can check that $\mathbb{E}\Vert\eta_0^N-\eta_0^{N+p}\Vert_{H^{-3}}^2$ does not converge to $0$ as $N\to\infty$. Thus there does not exist a realization $\tilde{\eta}_0$ with the same distribution as $\eta_0$, such that $\mathbb{E}\Vert\eta_0^N-\tilde{\eta}_0\Vert_{H^{-3}}\to 0$ as $N\to\infty$. Therefore, the estimate \eqref{eq:quantitative2} cannot give any rate for the convergence of $\eta^N$ to $\eta$. To get a convergence rate, we can only estimate the difference between the expectations of some functions related to $\eta^N$ and $\eta$ respectively, which does not depend on the joint distribution of $\eta^N$ and $\eta$. This idea gives rise to the estimate \eqref{eq:quantitative} in Theorem~\ref{thm:quantitative}.
\begin{proof}[Proof of Theorem~\ref{thm:quantitative}]
    From the first equality of \eqref{eq:weak-etatN}, we know that
    \begin{align*}
        \langle\varphi_i,\eta_{t+1}^N\rangle=&\langle\varphi_i\circ F\circ\phi_{\mu_N(t)},\eta_t^N\rangle+\Delta\langle(\varphi_i\circ F)'(\phi_{\bar{\rho}_t})\cdot h\ast\eta_t^N,\bar{\rho}_t\rangle\\
        &\hspace{0.5cm}+\frac{\Delta^2\sqrt{N}}{2}\langle(\varphi_i\circ F)''(\xi)\cdot[h\ast(\mu_N(t)-\bar{\rho}_t)]^2,\bar{\rho}_t\rangle\\
        =&\langle\varphi_i\circ F\circ\phi_{\bar{\rho}_t},\eta_t^N\rangle+\Delta\langle(\varphi_i\circ F)'(\phi_{\bar{\rho}_t})\cdot h\ast\eta_t^N,\bar{\rho}_t\rangle\\
        &\hspace{0.5cm}+\langle\varphi_i\circ F\circ\phi_{\mu_N(t)}-\varphi_i\circ F\circ\phi_{\bar{\rho}_t},\eta_t^N\rangle\\
        &\hspace{0.5cm}+\frac{\Delta^2\sqrt{N}}{2}\langle(\varphi_i\circ F)''(\xi)\cdot[h\ast(\mu_N(t)-\bar{\rho}_t)]^2,\bar{\rho}_t\rangle\\
        =&\langle Q_{t,t+1}\varphi_i,\eta_t^N\rangle+Z_t(\varphi_i),\quad 1\leq i\leq m,0\leq t\leq T-1,
    \end{align*}
    where $Z_t(\varphi_i)$ denotes the error terms:
    $$Z_t(\varphi_i)=\langle\varphi_i\circ F\circ\phi_{\mu_N(t)}-\varphi_i\circ F\circ\phi_{\bar{\rho}_t},\eta_t^N\rangle+\frac{\Delta^2\sqrt{N}}{2}\langle(\varphi_i\circ F)''(\xi)\cdot[h\ast(\mu_N(t)-\bar{\rho}_t)]^2,\bar{\rho}_t\rangle.$$
    By the same argument used to estimate $\I$ and $\IV$ in the proof of Proposition~\ref{prop:identify}, we can show that
    $$|Z_t(\varphi_i)|\leq C\sqrt{N}W_1^2(\mu_N(t),\bar{\rho}_t)\|\varphi_i\|_{C^2}.$$
    Thus similarly by Theorem~\ref{thm:Dobrushin} and Lemma~\ref{lem:expectation-Wasserstein}, we have
    $$\mathbb{E}|Z_t(\varphi_i)|\leq\frac{C\mathrm{e}^{Ct}}{\sqrt{N}}\|\varphi_i\|_{C^2}.$$
    By induction, we have for $1\leq i\leq m$ and $0\leq t\leq T$ that
    $$\langle\varphi_i,\eta_t^N\rangle=\langle Q_{0,t}\varphi_i,\eta_0^N\rangle+\sum_{s=0}^{t-1}Z_s(Q_{s+1,t}\varphi_i).$$
    And recall that $\langle\varphi_i,\eta_t\rangle=\langle Q_{0,t}\varphi_i,\eta_0\rangle$, so we can estimate that
    \begin{align*}
        &\left|\mathbb{E}\Psi\left(\langle\varphi_1,\eta_t^N\rangle,\dots,\langle\varphi_m,\eta_t^N\rangle\right)-\mathbb{E}\Psi\left(\langle\varphi_1,\eta_t\rangle,\dots,\langle\varphi_m,\eta_t\rangle\right)\right|\\
        =&\Bigg|\mathbb{E}\Psi\left(\langle Q_{0,t}\varphi_1,\eta_0^N\rangle+\sum_{s=0}^{t-1}Z_s(Q_{s+1,t}\varphi_1),\dots,\langle Q_{0,t}\varphi_m,\eta_0^N\rangle+\sum_{s=0}^{t-1}Z_s(Q_{s+1,t}\varphi_m)\right)\\
        &\hspace{0.5cm}-\mathbb{E}\Psi\left(\langle Q_{0,t}\varphi_1,\eta_0\rangle,\dots,\langle Q_{0,t}\varphi_m,\eta_0\rangle\right)\Bigg|\\
        \leq&\left|\mathbb{E}\Psi\left(\langle Q_{0,t}\varphi_1,\eta_0^N\rangle,\dots,\langle Q_{0,t}\varphi_m,\eta_0^N\rangle\right)-\mathbb{E}\Psi\left(\langle Q_{0,t}\varphi_1,\eta_0\rangle,\dots,\langle Q_{0,t}\varphi_m,\eta_0\rangle\right)\right|\\
        &\hspace{0.5cm}+C\sum_{i=1}^{m}\sum_{s=0}^{t-1}\mathbb{E}|Z_s(Q_{s+1,t}\varphi_i)|.
    \end{align*}
    Note that by the definition of time evolution operators, we have $\| Q_{t,t+1}\varphi\|_{C^2}\leq C\|\varphi\|_{C^2}$ and thus $\|Q_{s,t}\varphi\|_{C^2}\leq\mathrm{e}^{C(t-s)}\|\varphi\|_{C^2}$. Hence,
    $$\sum_{i=1}^{m}\sum_{s=0}^{t-1}\mathbb{E}|Z_s(Q_{s+1,t}\varphi_i)|\leq\sum_{i=1}^{m}\sum_{s=0}^{t-1}\frac{C\mathrm{e}^{Cs}}{\sqrt{N}}\| Q_{s+1,t}\varphi_i\|_{C^2}\leq\frac{C\mathrm{e}^{Ct}}{\sqrt{N}}.$$
    For the first term, it converges to $0$ by the classical multivariate Central Limit Theorem, and the convergence rate is given by the multivariate Berry-Esseen/Stein estimate\cite{Got91}:
    \begin{align*}
        &\left|\mathbb{E}\Psi\left(\langle Q_{0,t}\varphi_1,\eta_0^N\rangle,\dots,\langle Q_{0,t}\varphi_m,\eta_0^N\rangle\right)-\mathbb{E}\Psi\left(\langle Q_{0,t}\varphi_1,\eta_0\rangle,\dots,\langle Q_{0,t}\varphi_m,\eta_0\rangle\right)\right|\\
        \leq&\frac{C\|\Psi\|_{C_b^3}}{\sqrt{N}}\sum_{i=1}^{m}\langle|Q_{0,t}\varphi_i|,\bar{\rho}_0\rangle^3\leq\frac{C\mathrm{e}^{Ct}}{\sqrt{N}}.
    \end{align*}
    Combining the estimates above, the desired quantitative result \eqref{eq:quantitative} follows.
\end{proof}
\section*{\bf Acknowledgements}
The author would like to thank his advisor, Zhenfu Wang, for many helpful discussions on the applications of relative entropy method on both propagation of chaos and Gaussian fluctuations, as well as for valuable guidance on the writing of this paper. The author also thanks Matteo Tanzi for inspiring talks and discussions on globally coupled maps and self-consistent operators. This work was supported by the National Key R\&D Program of China (Project No.~2024YFA1015500) and by the NSFC (Grant Nos.~12595282 and 12171009).

\bibliographystyle{abbrv}
\bibliography{reference}

\end{document}